\documentclass[12pt]{amsart} 
\usepackage{amssymb,amsmath,amsthm,amsfonts}
\usepackage {latexsym,enumerate}
\usepackage{bbm}

\include{srctex}

\usepackage{marginnote}
\usepackage[letterpaper, top=1.2in, bottom=1.2in, outer=1.1in, inner=1.1in, heightrounded, marginparwidth=2.5cm, marginparsep=2mm]{geometry}

\allowdisplaybreaks
\newcommand{\nc}{\newcommand}
\nc{\les}{\lesssim}
\nc{\nit}{\noindent}
\nc{\nn}{\nonumber}
\nc{\D}{\partial}
\nc{\diff}[2]{\frac{d #1}{d #2}}
\nc{\diffn}[3]{\frac{d^{#3} #1}{d {#2}^{#3}}}
\nc{\pdiff}[2]{\frac{\partial #1}{\partial #2}}
\nc{\pdiffn}[3]{\frac{\partial^{#3} #1}{\partial{#2}^{#3}}}
\nc{\abs}[1] {\lvert #1 \rvert}
\nc{\cAc}{{\cal A}_c}
\nc{\cE}{{\cal E}}
\nc{\cF}{{\cal F}}
\nc{\cP}{{\cal P}}
\nc{\cV}{{\cal V}}
\nc{\cQ}{{\cal Q}}
\nc{\cGin}{{\cal G}_{\rm in}}
\nc{\cGout}{{\cal G}_{\rm out}}
\nc{\cO}{{\cal O}}
\nc{\Lav}{{\cal L}_{\rm av}}
\nc{\cL}{{\cal L}}
\nc{\cB}{{\cal B}}
\nc{\cZ}{{\cal Z}}
\nc{\cR}{{\cal R}}
\nc{\cT}{{\cal T}}
\nc{\cY}{{\cal Y}}
\nc{\cX}{{\cal X}}
\nc{\cXT}{{{\cal X}(T)}}
\nc{\cBT}{{{\cal B}(T)}}
\nc{\vD}{{\vec \mathcal{D}}}
\nc{\efield}{\mathcal{E}}
\nc{\vE}{{\vec \efield}}
\nc{\vB}{{\vec \mathcal{B}}}
\nc{\vH}{{\vec \mathcal{H}}}
\nc{\mR}{\mathcal{R}}
\nc{\mG}{\mathcal{G}}
\nc{\ty}{{\tilde y}}
\nc{\tu}{{\tilde u}}
\nc{\tV}{{\tilde V}}
\nc{\Pc}{{\bf P_c}}
\nc{\bx}{{\bf x}}
\nc{\bX}{{\bf X}}
\nc{\bXYZ}{{\bf XYZ}}
\nc{\oo}{\"o}
\nc{\bY}{{\bf Y}}
\nc{\bF}{{\bf F}}
\nc{\bS}{{\bf S}}
\nc{\dV}{{\delta V}}
\nc{\dE}{{\delta E}}
\nc{\TT}{{\Theta}}
\nc{\dPsi}{{\delta\Psi}}
\nc{\order}{{\cal O}}
\nc{\Rout}{R_{\rm out}}
\nc{\eplus}{e_+}
\nc{\eminus}{e_-}
\nc{\epm}{e_\pm}
\nc{\eps}{\varepsilon}
\nc{\vnabla}{{\vec\nabla}}
\nc{\G}{\Gamma}
\nc{\w}{\omega}
\nc{\mh}{h}
\nc{\mg}{g}
\nc{\vphi}{\varphi}
\nc{\tlambda}{\tilde\lambda}
\nc{\be}{\begin{equation}}
\nc{\ee}{\end{equation}}
\nc{\ba}{\begin{eqnarray}}
\nc{\ea}{\end{eqnarray}}

\nc{\g}{\gamma}
\nc{\ol}{\overline}

\newtheorem{theorem}{Theorem}[section]
\newtheorem{lemma}[theorem]{Lemma}
\newtheorem{prop}[theorem]{Proposition}
\newtheorem{corollary}[theorem]{Corollary}
\newtheorem{defin}[theorem]{Definition}

\nc{\pT}{\partial_T}
\nc{\pz}{\partial_z}
\nc{\pt}{\partial_t}
\nc{\la}{\langle}
\nc{\ra}{\rangle}
\nc{\infint}{\int_{-\infty}^{\infty}}
\nc{\halfwidth}{6.5cm}
\nc{\figwidth}{10cm}
\newcommand{\f}{\frac}

\nc{\nlayers}{L} \nc{\nsectors}{M}
\nc{\indicator}{\mathbf{1}}
\nc{\Rhole}{R_{\rm hole}}
\nc{\Rring}{R_{\rm ring}}
\nc{\neff}{n_{\rm eff}}
\nc{\Frem}{F_{\rm rem}}
\nc{\R}{\mathbb R}
\nc{\C}{\mathbb C}
\nc{\Z}{\mathbb Z}
\nc{\N}{\mathbb N}
\nc{\DD}{\Delta}
\nc{\cD}{\mathcal D}
\nc{\lnorm}{\left\|}
\nc{\rnorm}{\right\|}
\nc{\rnormp}{\right\|_{\ell^{p,\eps}}}
\nc{\rar}{\rightarrow}
\begin{document}

\begin{abstract}

	We investigate $L^1\to L^\infty$ dispersive estimates for the massless two dimensional Dirac  equation with a potential.  In particular, we show that the Dirac evolution satisfies the natural $t^{-\f12}$ decay rate, which may be improved to $t^{-\f12-\gamma}$ for any $0\leq \gamma<\frac{3}{2}$ at the cost of spatial weights. 
	We classify the structure of threshold obstructions as being composed of a two dimensional space of p-wave resonances and a finite dimensional space of eigenfunctions at zero energy.  We show that, in the presence of a threshold resonance, the Dirac evolution satisfies the natural decay rate except for a finite-rank piece.  While in the case of a threshold eigenvalue only, the natural decay rate is preserved.  In both cases we show that the decay rate may be improved at the cost of spatial weights.
	
\end{abstract}

\title[The Massless Dirac Equation in Two Dimensions]{\textit{The Massless Dirac Equation in Two Dimensions: Zero-Energy Obstructions and Dispersive Estimates}}

\author{M. Burak Erdo\smash{\u{g}}an, Michael Goldberg, William R. Green}
\thanks{The first author was partially supported by NSF grant DMS-1501041.
The second author is supported by Simons Foundation Grant 281057. The third author is supported by Simons Foundation Grant 511825.}  
\address{Department of Mathematics \\
University of Illinois \\
Urbana, IL 61801, U.S.A.}
 \email{berdogan@math.uiuc.edu}

\address{Department of Mathematics\\
University of Cincinnati \\
Cincinnati, OH 45221 U.S.A.}
\email{goldbeml@ucmail.uc.edu}
\address{Department of Mathematics\\
Rose-Hulman Institute of Technology \\
Terre Haute, IN 47803, U.S.A.}
\email{green@rose-hulman.edu} 

\maketitle
\section{Introduction}

We consider the linear Dirac equation with a potential:
\begin{align}\label{eqn:Dirac}
	i\partial_t \psi(x,t)=(D_m+V(x))\psi(x,t), \qquad
	\psi(x,0)=\psi_0(x).
\end{align}
Here  the spatial
variable $x\in \mathbb R^2$, and $\psi(x,t) \in \mathbb C^{2 }$.  The free Dirac operator
$D_m$ is defined by
\begin{align}\label{eqn:Dmdef}
	D_m=-i\alpha \cdot \nabla +m\beta =-i \sum_{k=1}^{2}\alpha_k \partial_{k}+m\beta
\end{align}
where $m\geq0$ is a constant, and 
the $2\times 2$ Hermitian matrices $\alpha_0:=\beta$ and $\alpha_j$ satisfy
\be \label{eqn:anticomm}
		\alpha_j \alpha_k+\alpha_k\alpha_j =2\delta_{jk}
		\mathbbm 1_{\mathbb C^{2}}, \,\,\,\,\,\,\,
		  j,k \in\{0, 1,2\}.
\ee
We consider the massless case, when $m=0$.  
For concreteness,  
we use
\begin{align}
	\beta=\left(\begin{array}{cc} 1& 0\\ 0 & -1
	\end{array}\right), \qquad \alpha_1=\left(\begin{array}{cc} 0 & 1\\ 1 & 0
	\end{array}\right), \qquad
	\alpha_2=\left(\begin{array}{cc} 0 & -i\\ i & 0
	\end{array}\right).
\end{align}
There is much interest in the massless case due to its connection to graphene, see \cite{FW} for example.
The Dirac equation was derived by Dirac as an attempt to connect the theories of quantum mechanics and special relativity.  Dirac's derivation allowed for a model that is first order in time, as required for quantum mechanical interpretations while having a finite speed of propagation and allowing for external fields in a relativistically invariant manner.  For a broader introduction to the Dirac equation, we refer the reader to the excellent text of Thaller, \cite{Thaller}.

The following identity,\footnote{Here and throughout the paper, scalar operators such as  $-\Delta+m^2-\lambda^2$ are understood as $(-\Delta+m^2-\lambda^2)\mathbbm 1_{\mathbb C^{2}}$.  Similarly, we denote $L^p(\R^2)\times L^p(\R^2)$ as $L^p(\R^2)$.}  which follows from   \eqref{eqn:anticomm},
\be  \label{dirac_schro_free}
	(D_m-\lambda \mathbbm 1)(D_m+\lambda \mathbbm 1) =(-i\alpha\cdot \nabla +m\beta -\lambda \mathbbm 1)
	(-i\alpha\cdot \nabla+m\beta+\lambda \mathbbm 1)   =(-\Delta+m^2-\lambda^2) 
\ee
allows us to formally define the free Dirac resolvent
operator $\mathcal R_0(\lambda)=(D_m-\lambda)^{-1}$ in terms of the
free resolvent $R_0(\lambda)=(-\Delta-\lambda)^{-1}$ of  the Schr\"odinger operator for $\lambda$ in the resolvent set:
\begin{align}\label{eqn:resolvdef}
	\mathcal R_0(\lambda)=(D_m+\lambda) R_0(\lambda^2-m^2).
\end{align}
For the massless equation, when $m=0$, we have
$$
	\mathcal R_0(\lambda)=(-i\alpha \cdot \nabla+\lambda) R_0(\lambda^2):=(D_0+\lambda)R_0(\lambda^2).
$$
Much of the analysis in this paper will be based on properties of $\mathcal R_0(\lambda)$ as $\lambda \to 0 $.
It should be emphasized that while the Dirac and Schr\"odinger resolvents are closely related by~\eqref{eqn:resolvdef}, the massless Dirac operator has very different behavior from the massive Dirac or Schr\"odinger operators in the low energy regime.  For example, $\mathcal R_0(0)$ exists as a well-defined operator while $R_0(\lambda^2)$ has a logarithmic singularity at the origin and the resolvent of a massive Dirac operator has a logarithmic singularity at the threshold $\lambda = \pm m$.  These differences carry over into the low-energy asymptotic structure of resolvents of $D_0 + V(x)$, which is again distinct from the threshold expansions for either Schr\"odinger or massive Dirac operators, \cite{eg2,egd}.

Detailed asymptotic expansions for the resolvents of both $D_0$ and its perturbations are computed in Section~\ref{sec:exps}. For certain choices of potential, the operator $D_0 + V(x)$ has an eigenvalue at zero.  It is also possible for zero to be a non-regular point of the spectrum without an eigenvalue present, a phenomenon known as a resonance.  We classify zero energy resonances and eigenvalues in terms of distributional solutions to $H\psi=0$ in Section~\ref{sec:TC}.  We say that zero energy is regular if there are no distributional solutions to $H\psi=0$ with $\psi\in L^\infty(\R^2)$, which may also be characterized by the uniform boundedness of the perturbed resolvent $(D_0+V-\lambda)^{-1}$ as $\lambda \to 0$.  We show that the classification of resonances for the massless Dirac equation and their dynamical consequences do not follow the same patterns as the Schr\"odinger equation.  

Before stating the dynamical results, we introduce some notation that will be used throughout the paper.  The function $\chi(\lambda)$ will denote a smooth, even cut-off around the origin in $\R$.  That is, $\chi(\lambda)=1$ if $|\lambda|<\lambda_1$ and $\chi(\lambda)=0$ if $|\lambda|>2\lambda_1$ for a sufficiently small, fixed constant $\lambda_1>0$.  The complementary cut-off is $\widetilde{\chi} = 1 - \chi$.  We use the notation $\la y\ra:=(1+|y|)^{\f12}$, and write $H:=D_0+V$ for the perturbed Dirac operator.  We also write $|V(x)|\les \la x\ra^{-\beta}$ to indicate that the entries of the potential all satisfy $|V_{ij}(x)|\les \la x\ra^{-\beta}$, $1\leq i,j\leq 2$, where $A\leq B$ denotes that there is an absolute constant $C$ so that $A\leq CB$.  We define the weighted spaces $L^{1,\gamma}=\{ f\,:\, \la \cdot \ra^{\gamma} f \in L^1(\mathbb R^2) \}$, and $L^{\infty,-\gamma}=\{ f\,:\, \la \cdot \ra^{-\gamma} f \in L^\infty(\mathbb R^2) \}$.
Our main results are the following small energy bounds:

\begin{theorem}\label{thm:main}
	
	Assume that $V$ is self-adjoint and $|V(x)| \les \la x\ra^{-\beta}$. 
	\begin{enumerate}[i)]
		
		\item Assume that zero is regular. If $\beta>2$,  then
		$$
		\| e^{-itH}\chi(H)  \|_{L^1\to L^\infty} \les \la t\ra^{-\f12}.
		$$
		Further, for $0\leq \gamma <\frac32$, if $\beta>2+2\gamma$, then
		$$
		\| e^{-itH} \chi(H)\|_{L^{1,\gamma}\to L^{\infty,-\gamma}} \les \la t\ra^{-\f12-\gamma}.
		$$
		
		\item If zero is not regular, then for fixed $0\leq \gamma<\tfrac12$, 
		$$
		\| e^{-itH}P_{ac} \chi(H) -F_t\|_{L^{1,\gamma}\to L^{\infty,-\gamma}}\les \la t\ra^{-\f12-\gamma},
		$$
		provided that $\beta>3+2\gamma$. Here $F_t$ is a finite-rank operator, which satisfies the bounds $ \sup_t \|F_t\|_{L^1\to L^\infty}\les 1 $ and if $|t|>2$ one has $\|F_t\|_{L^1\to L^\infty}\les (\log |t|)^{-1}$.
		
		\item If there is only an eigenvalue  at zero, then $F_t=0$.
		
	\end{enumerate}
	
\end{theorem}

We emphasize that our main results are the low energy bounds presented above. We also provide an explicit construction of the operator $F_t$, see \eqref{eqn:Ft def} below.
For the sake of completeness, we include the high energy result stated below.

\begin{theorem}\label{thm:main_hi}
	
	Assuming $V$ is self-adjoint, has continuous entries satisfying $|V(x)| \les \la x\ra^{-\beta}$ and there are no embedded eigenvalues in the real line.
	If   $\beta>2$, then
		$$
		\| e^{-itH} \widetilde\chi(H)\la H \ra^{-2-}\|_{L^1\to L^\infty} \les \la t\ra^{-\f12}.
		$$
		Further,  if $0\leq \gamma \leq \f32$ and $\beta>\min(2+2\gamma,3) $, we have
		$$
		\| e^{-itH} \widetilde\chi(H) \la H \ra^{-2-}\|_{L^{1,\gamma}\to L^{\infty,-\gamma}} \les \la t\ra^{-\f12-\gamma}.
		$$ 
\end{theorem}

We note that the assumption of a lack of embedded eigenvalues is not needed for our low energy results in Theorem~\ref{thm:main}, as the spectral properties in a neighborhood of zero are dictated by the threshold behavior.  The lack of embedded eigenvalues has been established in the massive case, \cite{BC1}, and in the massless case for a sufficiently small potential, \cite{CGetal}.

We establish the dispersive bounds by employing the functional calculus for the Dirac operator.  For the class of potentials we consider, $H$ is self-adjoint and the spectrum of $H$ coincides with the real line. Under these circumstances, see \cite{RS1}, the Stone's formula for spectral measures yields: 
\be\label{Stone}
	e^{-itH}P_{ac}(H)= \frac{1}{2\pi i}  \int_{\R} e^{-it\lambda}\big[ \mathcal R_V^+(\lambda)-\mR_V^-(\lambda) \big]f\, d\lambda
\ee
Here the perturbed resolvents are $\mR_V^\pm(\lambda)=\lim_{\epsilon\to 0^+} (D_0+V-(\lambda \pm i\epsilon))^{-1}$, and their difference provides the spectral measure.  We take advantage of the identity \eqref{eqn:resolvdef} to develop the spectral measure from Schr\"odinger resolvents. The Schr\"odinger free resolvent $$R_0^\pm(\lambda ^2)=\lim_{\epsilon\to 0^+}(-\Delta -(\lambda ^2\pm i\epsilon))^{-1}$$ and the perturbed Schr\"odinger resolvent operators $$R_V^\pm(\lambda ^2)=\lim_{\epsilon\to 0^+}(-\Delta+V -(\lambda^2\pm i\epsilon))^{-1}$$
are well-defined as operators between weighted $L^2(\R^2)$ spaces, see \cite{agmon}.  

To the authors' knowledge, this is the first study of dispersive estimates for the two dimensional massless Dirac equation.  A recent paper of Cacciafesta and Ser\'e, \cite{CS} investigated local smoothing estimates for the massless Dirac equation in dimensions two and three.   The massive Dirac has been studied by the first and third author, \cite{egd}, with Toprak \cite{EGT2d}.  The three-dimensional massive Dirac equation is more studied going back to the work of Boussaid \cite{Bouss1}, and D'Ancona and Fanelli, \cite{DF}.   The characterization of threshold obstructions and their effect on the dispersive bounds have recently been studied by the first and third author and Toprak, \cite{EGT}.  Much of the work has roots in the study of other dispersive equations, notably the Schr\"odinger \cite{Mur,Sc2,ES,eg2,eg3,EGG4,ebru} and wave \cite{DF,Gwave,beceanu} equations.

Our low energy results in Theorem~\ref{thm:main} establish the natural time decay $\la t \ra^{-\f12}$ for the Dirac evolution while assuming less decay of the potential than has been required in the massive case.  The improvement comes from using a more delicate argument based on Lipschitz continuity of the spectral measure, rather than direct integration by parts in the Stone's formula. A similar argument was used in \cite{eg3}.

In addition, this is the first result in which all the slow time decay caused by a p-wave resonance is controlled in a finite rank term.  Previous works on the Schr\"odinger or wave equation, \cite{Mur,eg2,Gwave}, did not observe this asymptotic structure.  Even in the weighted $L^2$ setting, \cite{Mur}, finite rank leading order terms had an error whose decay was only logarithmically better.  The method we  develop for computing spectral measures here can recover an analogous result (finite rank leading order, with polynomial decay of the remainder) for the Schr\"odinger evolution as well.

There is also much interest in the study of non-linear Dirac equations.  See \cite{EV,BH,CTS,BC2} for example.  There is a longer history in the study of spectral properties of Dirac operators.  Limiting absorption principles for the Dirac operators have been studied in \cite{Yam,GM,EGG,CGetal}.  In particular, the recent work \cite{EGG} of the authors applies in all dimensions $n\geq 2$ for both massive and massless equations, while the recent work of Carey, et.~ al.~\cite{CGetal} applies to massless equations.  The lack of embedded eigenvalues, singular continuous spectrum and other spectral properties is well established, \cite{BG1,GM,MY,CGetal,BC1}.  In particular, for the class of potentials we consider, the Weyl criterion implies that $\sigma_{ac}(H)=\sigma(D_0)=(-\infty,\infty)$.   There are no embedded eigenvalues provided the potential is small, see Theorem 3.15 in \cite{CGetal}.

The paper is organized as follows.  We begin by proving the natural dispersive estimates for the free massless Dirac operation in Section~\ref{sec:free}.  In Section~\ref{sec:exps} we develop a variety of expansions for the free resolvent that will be needed to study the spectral measure in \eqref{Stone}.  In Section~\ref{sec:zeroregular} we prove Theorem~\ref{thm:main} when zero energy is regular.  In Section~\ref{sec:resolv notfree} we establish more delicate expansions of the perturbed resolvent around the threshold in the presence of resonances and/or eigenvalues so that we may prove Theorem~\ref{thm:main} when the threshold is not regular in Section~\ref{sec:disp notfree}.  In Section~\ref{sec:TC} we provide a characterization of the threshold obstructions that relates them naturally to the various subspaces of $L^2$ that arise in the resolvent expansions. Section~\ref{sec:high energy} provides the high energy estimates to prove Theorem~\ref{thm:main_hi}. Finally, Section~\ref{sec:int ests} contains the various integral estimates needed throughout the paper.

\section{Free Dirac dispersive estimates}\label{sec:free}

Due to the relationship between the massless free Dirac evolution and the
free wave equation, $D_0^2 f=-\Delta f$,
we can expect a natural time decay rate of size $|t|^{-\f12}$ 
as one has in the wave equation (when $m=0$)  provided the initial data has more than $\f 32$ weak derivatives in $L^1(\R^2)$.
In the case of Dirac equation, as in Schr\"odinger equation, the time decay can be improved  at the cost of spatial weights.   

\begin{theorem}\label{thm:free} 
	
	We have the estimate
	$$
		\|  e^{-it D_0} \la D_0 \ra^{-\f32-} 
		\|_{L^1\to L^\infty} \les  t^{-\f12}.
	$$	
	Further, one has
	$$
		\| \la x\ra^{-\gamma}  e^{-it D_0} \la D_0 \ra^{-2-}  \la y\ra^{-\gamma}
		\|_{L^1\to L^\infty} \les \la t\ra^{-\f12-\gamma},
	$$	
	for any $0\leq \gamma \leq \tfrac32$.
	
\end{theorem}

The proof of this theorem is based on asymptotic expansions of the spectral measure of the free Dirac operator, both at low energies and high energies. To best utilize these expansions, we employ the notation
$$
f(\lambda )=\widetilde O(g(\lambda ))
$$
to denote
$$
\frac{d^j}{d\lambda ^j} f = O\big(\frac{d^j}{d\lambda ^j} g\big),\,\,\,\,\,j=0,1,2,3,...
$$
The notation primarily refers to derivatives with respect to the spectral variable $\lambda $  in the expansions for
the integral kernel of the free resolvent operator.  In the context of \eqref{eqn:resolvdef}, due to the gradient, we use the $\widetilde O(g)$ to refer to $|x-y|$ as well.
If the derivative bounds hold only for the first $k$ derivatives we  write $f=\widetilde O_k (g)$.  In addition, if we write $f=\widetilde O_k(1)$, we mean that differentiation up to order $k$ is comparable to division by $\lambda $ and/or $|x-y|$.
This notation applies to operators as well as
scalar functions; the meaning should be clear from the
context.

\begin{proof}[Proof of Theorem \ref{thm:free}]  
First note that   in the free case the Stone's formula, \eqref{Stone}, is
\begin{align}\label{eq:Stone}
	e^{-itD_0} =\int_{\R}  e^{-it\lambda}[\mathcal R_0^+ -\mR_0^-](\lambda)(x,y) \, d\lambda.
\end{align}
We consider the low energy first. Using \eqref{eq:dr1}, the formula  $[R_0^+-R_0^-](\lambda^2)(x,y)=\frac{i}{2}J_0(\lambda |x-y|)$, and
the asymptotics for the Bessel functions, see \cite{AS,Sc2},  we can write 
\begin{multline}\label{R0pm}
[\mathcal R_0^+ -\mR_0^-](\lambda)(x,y)=
	(-i\alpha \cdot \nabla +\lambda)[R_0^+
	-R_0^-](\lambda^2)(x,y) \\ =	\left\{\begin{array}{ll}\frac{i 
	\lambda}2-\frac{\lambda^2}4 \alpha\cdot(x-y)  +\widetilde O_2( \lambda^3 |x-y|^2 ), & |\lambda|\,|x-y|\ll 1\\
	e^{i\lambda|x-y|}\widetilde \omega_+(\lambda|x-y|)
	+e^{-i\lambda|x-y|}\widetilde \omega_-(\lambda |x-y|), & |\lambda|\,|x-y|\gtrsim 1
	\end{array}
	\right.
\end{multline}
where  $\widetilde \omega_\pm(\lambda|x-y|)$ satisfies 
$$
	\widetilde \omega_\pm(\lambda|x-y|)=  \widetilde O \big(|\lambda| (1 + |\lambda| |x-y|)^{-\f12}\big). 
$$

 Let  $\mu_0(\lambda)(x,y):=\chi(\lambda)[\mathcal R_0^+ -\mR_0^-](\lambda)(x,y)$.
The formula \eqref{R0pm} implies that 
\be \label{R0pm0}
|\mu_0(\lambda)(x,y)|\les  |\lambda| (1+ |\lambda||x-y|)^{-\frac12}, 
\ee
\be \label{R0pm1}
|\partial_\lambda \mu_0(\lambda)(x,y)|\les  (1+ |\lambda||x-y|)^{\frac12}, 
\ee
\be \label{R0pm2}
|\partial^2_\lambda \mu_0(\lambda)(x,y)|\les  |x-y| (1+ |\lambda||x-y|)^{\frac12}.
\ee
Thus, using \eqref{R0pm0} and \eqref{R0pm1} we have 
\be\label{mu0C12}
|\mu_0(\lambda_1)(x,y)-\mu_0(\lambda_2)(x,y)|\les |\lambda_1-\lambda_2|^\frac12  |\lambda_2|^{\frac12},
\ee
for $|\lambda_1|\leq|\lambda_2|\les 1$.  To obtain this consider the cases $|\lambda_1-\lambda_2| \approx |\lambda_2|$ and $|\lambda_1-\lambda_2| \ll |\lambda_2|$
 separately. In the former case  the bound follows from \eqref{R0pm0}. In the latter case, the mean value theorem and \eqref{R0pm1} give  the bound  $|\lambda_1-\lambda_2|   (1+ |\lambda_2||x-y|)^{\frac12}$. Interpolating this with \eqref{R0pm0}  and noting that $|\lambda_1|\approx|\lambda_2|$,  we obtain \eqref{mu0C12}.  
 
 We also state two other bounds for $\mu_0$ which will be useful in later sections.
 The interpolation argument above also implies that
\be\label{mu0C12alpha}
|\mu_0(\lambda_1)(x,y)-\mu_0(\lambda_2)(x,y)|\les |\lambda_1-\lambda_2|^{\frac12+\gamma}  |\lambda_2|^{\frac12-\gamma} \la x-y\ra^\gamma,\,\,\,\,0\leq \gamma \leq \frac12.
\ee
 Similarly, using \eqref{R0pm1} and \eqref{R0pm2} we obtain the bound
\be\label{mu0C32}
|\partial_\lambda \mu_0(\lambda_1)(x,y)-\partial_\lambda \mu_0(\lambda_2)(x,y)|\les |\lambda_1-\lambda_2|^\gamma |x-y|^{\gamma} (1+|\lambda_2||x-y|)^{\frac12},\,\,\,0\leq \gamma\leq 1.
\ee

Using the support of $\chi(\lambda)$ in the definition of $\mu_0$,
it is easy to see that
$$\Big|\int_{\R}  e^{-it\lambda}   \mu_0(\lambda)(x,y) \, d\lambda \Big|\les 1.$$
For $|t|\gtrsim 1$, again using the support of $\chi(\lambda)$ and \eqref{mu0C12}, we have 
\begin{multline}\label{lip12trick}
\Big|\int_{\R}  e^{-it\lambda}   \mu_0(\lambda)(x,y) \, d\lambda \Big|  =\frac12 \Big|\int_\R e^{-it\lambda} ( \mu_0(\lambda)(x,y)- \mu_0(\lambda-\frac\pi{t})(x,y))d\lambda\Big| \\ \les |t|^{-\frac12}
  \int_{-1}^1 1\, d\lambda \les |t|^{-\frac12}.
\end{multline}
For the weighted bounds,  after two integration by parts, we have
 \be  \label{alpha32}
\Big|\int_{\R}  e^{-it\lambda}   \mu_0(\lambda)  \, d\lambda \Big|  =\frac1{t^2}\Big|\int_{\R}  e^{-it\lambda} \partial^2_\lambda \mu_0 (\lambda)  \, d\lambda   \Big|    \les \frac1{t^2} \la x\ra^{\frac32} \la y\ra^{\frac32}. 
\ee
Interpolating these bounds we conclude for any $\gamma\in[0,\frac32]$ that
$$\Big|\int_{\R}  e^{-it\lambda} \chi(\lambda) \mu_0(\lambda)(x,y) \, d\lambda \Big|\les \la t \ra^{-\frac12} \Big(\frac{\la x\ra \la y\ra}{\la t\ra}\Big)^\gamma.$$

For large energies, to prove the first claim it suffices to bound
\be\label{eq:high free}
	\sup_{L\geq 1} \int_{-\infty}^{\infty}  e^{-it\lambda} \lambda^{-\f32-} \widetilde \chi(\lambda) \chi(\lambda/L) [\mR _0^+ -\mR _0^-](\lambda)(x,y)\, d\lambda.
\ee
Noting that $\mR _0^+ -\mR _0^-=[D_0+\lambda ]J_0(\lambda|x-y|)$ is comparable to $\lambda J_0(\lambda |x-y|)$, see \eqref{R0pm} and \cite{AS}. Using Lemmas~3.2 and 5.3 in \cite{Gwave}, we have the bounds
$$
	|\eqref{eq:high free}|\les \left\{
	\begin{array}{l}
		t^{-\f12}\\
		\frac{\la x\ra^{\f12} \la y \ra^{\f12}}{t^{\f32}}+\frac{1}{t^2}
	\end{array}
	\right. \qquad t>2.
$$
However these estimates rely on oscillation that may not be present when $t$ is small.  To obtain a uniform bound for small times, the integrand must be absolutely convergent.  Given the growth of $|\omega_\pm(\lambda|x-y|)|\les |\lambda| $, we need a multiplier that decays like $|\lambda|^{-2-}$ to conclude
$$
	\sup_{L\geq 1} \bigg|\int_{-\infty}^{\infty}  \lambda^{-2-} \widetilde \chi(\lambda) \chi(\lambda/L) [\mR _0^+ -\mR _0^-](\lambda)(x,y)\, d\lambda
	\bigg| \les 1
$$
uniformly in $x$ and $y$ for small $t$.  The additional powers of $\lambda$ correspond to extra mollification in the $x$ variable, using $\la D_0\ra^{-2-}$ instead of $\la D_0\ra^{-\frac32-}$.
\end{proof}

\section{Free resolvent expansions around zero energy}\label{sec:exps}

In this section we study the behavior of the free  Dirac resolvent more carefully by using the properties of   free Schr\"odinger resolvent $R_0(\lambda )=(-\Delta-\lambda )^{-1}$.    
Following \cite{Sc2,eg2,eg3,EGT2d}, we have the following expansion for the Schr\"odinger resolvent.  These results have their roots in work of Jensen and Nenciu, \cite{JN}.
\begin{lemma} \label{lem:schro_resol} Let $0<\lambda \ll 1$. For $\lambda |x-y|<1$, we have the expansions
	\begin{multline}\label{eq:freschroresolv}
		R_0^{\pm}(\lambda ^2) =g^\pm(\lambda )+G_0
		+  \widetilde 
		O_2(\lambda ^2|x-y|^2\log(\lambda |x-y|))\\
		 =g^\pm(\lambda )+G_0
		+g_1^\pm(\lambda ) G_1+\lambda ^2 G_2+ \widetilde 
		O_3(\lambda ^4|x-y|^4\log(\lambda |x-y|)),
	\end{multline}
 where
\begin{align}
	g^\pm(\lambda )&=-\frac1{2\pi} \big(\log(\lambda /2)+\gamma\big)\pm\frac{i}4 \label{g def}\\
	g_1^\pm(\lambda )&=-\frac{\lambda ^2}4g^\pm(\lambda )-\frac{\lambda ^2}{8\pi}  \label{g1 def}\\
	G_0f(x)&=-\frac{1}{2\pi}\int_{\R^2} \log|x-y|f(y)\,dy, \label{G0 def}\\
	G_1f(x)&=\int_{\R^2} |x-y|^2f(y)\, dy,\label{G1 def}\\
	G_2f(x)&= \frac1{8\pi}\int_{\R^2} |x-y|^2\log|x-y|f(x)\, dy.\label{G2 def}
\end{align}
For $\lambda |x-y|\gtrsim 1$, we have 
\be \label{R0large}
	R_0^\pm(\lambda ^2)(x,y)=   e^{\pm i\lambda |x-y|} \omega_\pm(\lambda |x-y|), \,\,\,\,\,\,
	|\omega_{\pm}^{(j)}(y)|\lesssim (1+|y|)^{-\frac{1}{2}-j},\,\,\,j=0,1,2,\ldots.
\ee
\end{lemma}
Using \eqref{eqn:resolvdef} we have
\be\label{eq:dr1}
	\mR_0^\pm(\lambda)=\left[-i\alpha \cdot \nabla+ \lambda I\right]R_0^\pm(\lambda^2)
\ee
We write (for $|\lambda|\, |x-y|\ll 1$)
\begin{multline}
 \label{eq:R0low}
	\mR_0^\pm(\lambda)=\mG_{0,0}+\lambda g^\pm(\lambda) \mG_{1,1}+\lambda \mG_{1,0}+g_1^\pm(\lambda) \mG_{2,1}+\lambda^2\mG_{2,0}\\
	  +    \widetilde O_2(\lambda^3 |x-y|^2 \log(\lambda |x-y|)  )
\end{multline}
where
\begin{align}
	&\mG_{0,0}=\mR_0(0)=-i\alpha \cdot \nabla  G_0(x,y)=\frac{i\alpha \cdot(x-y)}{2\pi |x-y|^2}\label{G00def}\\
	&\mG_{1,1}=1\label{G11def}\\
	&\mG_{1,0}=G_0(x,y)=-\frac{1}{2 \pi}\log|x-y|=(-\Delta)^{-1}(x,y) \label{G10def}    \\
	&\mG_{2,1}=-i\alpha \cdot \nabla G_1(x,y)=-2 i\alpha \cdot(x-y)\label{G21def}\\
	&\mG_{2,0}=-i\alpha \cdot \nabla G_2(x,y)=-\frac{i\alpha \cdot \nabla (|x-y|^2 \log |x-y|)}{8\pi }\label{G20def} 
	\end{align}
	The subscripts indicate the powers of $\lambda$ and $\log \lambda$ multiplying each operator in \eqref{eq:R0low}.

\begin{lemma}\label{lem:R0simpleee} We have
 \be\label{eq:R0simple}
\mR_0^\pm(\lambda)(x,y)=\mG_{0,0}+O\big(|\lambda| (1+ (|\lambda|\,|x-y|)^{0-})\big)=O\Big(\lambda+\frac1{|x-y|}\Big).
\ee
\be \label{R01}
|\partial_\lambda \mR_0^\pm(\lambda)(x,y)|\les   (|\lambda||x-y|)^{0-}   + (|\lambda||x-y|)^{\frac12}, 
\ee 
\be \label{R02}
|\partial^2_\lambda \mR_0^\pm(\lambda)(x,y)|\les  |\lambda|^{-1} (|\lambda||x-y|)^{0-} + |\lambda|^{\frac12}|x-y|^{\frac32}   .
\ee  
\end{lemma}  
\begin{proof}	The expansions follow from \eqref{eq:R0low} when   $|\lambda|\,|x-y|\ll 1$. Recall that when $|\lambda|\, |x-y|\gtrsim 1$, we have the representation
$$
 \mathcal R_0^\pm (\lambda)(x,y)= 
	e^{\pm i\lambda|x-y|}\widetilde \omega_\pm(\lambda|x-y|),$$
where  $\widetilde \omega_\pm(\lambda|x-y|)$ satisfies 
\be \label{eq:omegatilde}
	\widetilde \omega_\pm(\lambda|x-y|)=  \widetilde O \big(|\lambda| (1 + |\lambda| |x-y|)^{-\f12}\big). 
\ee
Also using \eqref{G00def}, the error in \eqref{eq:R0simple} is bounded by
$$
	\widetilde\chi(\lambda|x-y|)\big(|\lambda|^\frac12|x-y|^{-\frac12}  +|x-y|^{-1}\big)\les |\lambda|.
$$	 
The bounds \eqref{R01} and \eqref{R02} for $|\lambda|\, |x-y|\gtrsim 1$ follow similarly using 
the high energy representation of  $\mathcal R_0^\pm $ above.
\end{proof}

As a corollary we have the following Lipschitz bounds. The $\frac12$-Lipschitz bound cannot be improved without growth in $|x-y|$, which leads to weights in the dispersive bounds,   due  to the large $\lambda|x-y|$ term.
\begin{corollary}\label{cor:lipbounds} For $|\lambda_1|\leq |\lambda_2|\les 1 $, we have
\be\label{R0lip12}
\big|\mR_0^\pm(\lambda_1)(x,y)-\mR_0^\pm(\lambda_2)(x,y)\big|\les |\lambda_1-\lambda_2|^{\frac12} |\lambda_2|^{\frac12-} (1+|x-y|^{0-}),
\ee
and more generally 
\begin{multline}\label{R0lip12alpha}
\big|\mR_0^\pm(\lambda_1)(x,y)-\mR_0^\pm(\lambda_2)(x,y)\big|\les \\ |\lambda_1-\lambda_2|^{\frac12+\gamma} |\lambda_2|^{\frac12-\gamma-} (|x-y|^{\gamma}+|x-y|^{0-}), \,\,\,\,0\leq\gamma<\frac12.
\end{multline}
Moreover for each $0\leq \gamma\leq 1$, we have 
\begin{multline}\label{R0lip1alpha}
\big|\partial_{\lambda}\mR_0^\pm(\lambda_1)(x,y)-\partial_\lambda\mR_0^\pm(\lambda_2)(x,y)\big|\les \\ |\lambda_1-\lambda_2|^{\gamma} |\lambda_1|^{-\gamma-} \big(  |x-y|^{0-} +|x-y|^{\frac12+\gamma}\big),\,\,\,\,0\leq \gamma\leq 1.
\end{multline}
\end{corollary}
\begin{proof} Note that \eqref{R0lip12} follows from \eqref{R0lip12alpha} with $\gamma=0$. 
When $|\lambda_1-\lambda_2|\gtrsim  |\lambda_2| $, the bound \eqref{R0lip12alpha} follows from \eqref{eq:R0simple} since the leading term $\mG_{0,0}$ cancels out. When $|\lambda_1-\lambda_2|\ll |\lambda_2|$, using  the mean value theorem, \eqref{R01},  and noting that  $|\lambda_1|\approx|\lambda_2|$, we obtain
$$
\big|\mR_0^\pm(\lambda_1)(x,y)-\mR_0^\pm(\lambda_2)(x,y)\big|\les |\lambda_1-\lambda_2|
 \big((|\lambda_2||x-y|)^{0-}   + (|\lambda_2||x-y|)^{\frac12}\big).
$$
Also note that  when $|\lambda_2||x-y|\gtrsim 1$,
$$
\big|\mR_0^\pm(\lambda_1)(x,y)-\mR_0^\pm(\lambda_2)(x,y)\big|\les |\lambda_2|^\frac12  |x-y|^{-\frac12},
$$
 and when $|\lambda_2||x-y|\ll 1$ by \eqref{eq:R0simple}
$$
\big|\mR_0^\pm(\lambda_1)(x,y)-\mR_0^\pm(\lambda_2)(x,y)\big|\les |\lambda_2| \big(1+ (\lambda_2|x-y|)^{0-}\big).
$$
Interpolating these   bounds, we obtain   \eqref{R0lip12alpha}.
The proof of \eqref{R0lip1alpha} is   similar using \eqref{R01} and \eqref{R02}.
\end{proof}

In the case when zero is not regular, we will need a further expansion of $\mR_0^\pm$: 
\begin{lemma}\label{lem:R0E0}
 We have the expansion for the kernel of the free resolvent
	\be\label{eqn:R0E0}
		\mR_0^\pm(\lambda)(x,y)=\mG_{0,0}(x,y)+\lambda g^\pm(\lambda) \mG_{1,1}(x,y)+\lambda \mG_{1,0}(x,y)+E_0^\pm(\lambda)(x,y).
	\ee
	Further, when $|\lambda|\leq 1$, the error term satisfies  
$$
		|E_0^\pm(\lambda)(x,y)| \les 
			|\lambda|  (|\lambda|\la x-y\ra)^{k},\,\,\,0<k<1. 
$$
Moreover, for  $0\leq \gamma<\frac12$ and  $|\lambda_1|\leq |\lambda_2|\les 1$, we have
$$
|E_0^\pm(\lambda_2)(x,y)-E_0^\pm(\lambda_1)(x,y)|\les  |\lambda_1-\lambda_2|^{\frac12+\gamma}  |\lambda_2|^{\frac12-\gamma+k} \la x-y\ra^k,\,\,\,\tfrac12\leq k < 1.
$$	 
\end{lemma}
\begin{proof}
The first bound for the error term follows from \eqref{eq:R0low} when $|\lambda||x-y|\ll 1$.
When $|\lambda||x-y|\gtrsim 1$, it follows by writing
$$
|E_0^\pm(\lambda)(x,y)| \les  \frac{|\lambda|}{(1+|\lambda||x-y|)^\frac12}+|x-y|^{-1}+|\lambda| |\log(|\lambda||x-y|)\les |\lambda|  (|\lambda|\la x-y\ra)^{k}, 
$$
provided that $k>0$. 
Similarly,  note that when $|\lambda||x-y|\ll 1$ we have
$$
|\partial_{\lambda}E_0^\pm(\lambda)(x,y)|\les (|\lambda||x-y|)^k, \,\,\,0<k<1,
$$
and for $|\lambda||x-y|\gtrsim 1$ we have
$$
|\partial_{\lambda}E_0^\pm(\lambda)(x,y)| \les (|\lambda||x-y|)^\frac12 +  |\log(|\lambda||x-y|)\les  (|\lambda|\la x-y\ra)^{k}, \,\,\,k\geq \tfrac12.
$$

Using these bounds with $\tfrac12\leq k<1$, we obtain the Lipschitz bound by interpolating the trivial bound,
$$|E_0^\pm(\lambda_1)-E_0^\pm(\lambda_2)|\les |\lambda_2|  (|\lambda_2|\la x-y\ra)^{k},$$
with the bound we obtain using the mean value theorem:
$$|E_0^\pm(\lambda_1)-E_0^\pm(\lambda_2)|\les |\lambda_1-\lambda_2|  (|\lambda_2|\la x-y\ra)^{k}.
$$
\end{proof}

\section{Small energy dispersive estimates when zero is regular  }\label{sec:zeroregular}

As usual, see for example \cite{Sc2,eg2,egd,EGT,EGT2d}, we use the symmetric resolvent identity to understand the low energy evolution.  In the Dirac context the potentials are matrix-valued, and we have the assumption that  the matrix 
$V:\mathbb R^2 \to \mathbb C^{2\times 2}$ is self-adjoint.  Hence, we may use the spectral theorem to write
$$
V=B^*\left(\begin{array}{cc}
\zeta_1 & 0 \\ 0 &\zeta_2
\end{array}\right)B
$$
with $\zeta_j \in \mathbb R$.  To employ the symmetric identity, with $\eta_j =|\zeta_j|^{\f12}$, we write
\begin{align*}
V=B^*\left(\begin{array}{cc}
\eta_1 & 0 \\ 0 & \eta_2
\end{array}\right) U \left(\begin{array}{cc}
\eta_1 & 0 \\ 0 & \eta_2
\end{array}\right)B = v^*Uv,
\end{align*}	
where
\begin{align}\label{eq:vabcd} 
U=\left(\begin{array}{cc}
\textrm{sign}(\zeta_1) & 0 \\ 0 & \textrm{sign}(\zeta_2)
\end{array}
\right),\,\,\,\text{ and }\,\,
v=\left(\begin{array}{cc}
a & b\\ c &d
\end{array}\right):=\left(\begin{array}{cc}
\eta_1 & 0 \\ 0 & \eta_2
\end{array}\right)  B.
\end{align}
Note that the entries of $v$ are $\les \la x\ra^{-\beta/2}$, provided that the entries of $V$ are $\les \la x\ra^{-\beta}$. 
 
Define the operators
\be\label{eqn:M def}
	M^\pm (\lambda)=U+v\mR_0^\pm (\lambda)v^*,
\ee
and let
\be \label{eqn:T def}
	T:=U+v\mG_{0,0}v^*=M^\pm(0).
\ee

\begin{defin}\label{def:resonances}
	
	We make the following definitions that characterize zero energy obstructions.
	\begin{enumerate}[i)]
	\item We define zero energy to be regular if $T=M^\pm (0)$ is invertible on $L^2(\R^2)$.  
	
	\item 	We say there is a  resonance of the first kind at zero if $T$ is not invertible on $L^2$, but $S_1v\mG_{1,1}v^*S_1$ is invertible on $S_1L^2$ where $S_1$ is the Riesz projection onto the kernel of $T$.
	
	\item We say there is a resonance of the second kind at zero if $S_1v\mG_{1,1}v^*S_1$ is not invertible.
	
	\item Let $S_2$ be the Riesz projection onto the kernel of $S_1v\mG_{1,1}v^*S_1$, then $S_1-S_2$ has rank at most two and $S_1-S_2\neq 0$ corresponds to the existence of `p-wave' resonances at zero.  $S_2\neq 0$ corresponds to the existence of an eigenvalue at zero.  In contrast to the massive case, see \cite{egd}, there are no `s-wave' resonances in the massless case.  See Section~\ref{sec:TC} below for a complete characterization.
	
	\item Noting that $v\mG_{0,0}v^*$ is compact and self-adjoint,   $T=U+v\mG_{0,0}v^*$ is  a compact perturbation of $U$.  Since the spectrum of $U$ is in $\{\pm 1\}$, zero is an isolated point of the spectrum of $T$ and the kernel is finite dimensional.  It then follows that $S_1$ is a finite rank projection, and since $S_2\leq S_1$, so is $S_2$.
	
	\end{enumerate}
	
\end{defin}

We employ the following terminology from \cite{Sc2,eg2,eg3}:

\begin{defin}
	We say an operator $T:L^2(\R^2) \to   L^2(\R^2)$ with kernel
	$T(\cdot,\cdot)$ is absolutely bounded if the operator with kernel
	$|T(\cdot,\cdot)|$ is bounded from $  L^2(\R^2)$ to $ L^2(\R^2)$. 	
\end{defin}

We note that Hilbert-Schmidt and finite-rank operators are absolutely bounded operators.  Recall that the Hilbert-Schmidt norm of an integral operator $T$ with integral kernel $T(x,y)$ is defined by
$$
	\|T\|_{HS}^2=\int_{\mathbb R^4} |T(x,y)|^2\, dx\, dy.
$$

We now concentrate on the case when zero is regular. The following expansions for $M^\pm(\lambda)$ around zero energy suffice in this case.
\begin{lemma}\label{lem:Mexp1}
	Assume  that $|V(x)|\les \la x\ra^{-\beta}$.
	\begin{enumerate}[i)]
		\item 	If $\beta>2$, then
	\be\label{eq:Mpntwise}
		M^\pm(\lambda)=T+ O(\lambda^{1-}). 
	\ee
	\item If $\beta>2+2\gamma$ for some $0\leq \gamma<\frac12$, then for $0<|\lambda_1|\leq |\lambda_2|\les 1$, we have 
	\be\label{eq:Mlip12alpha}
	  M^\pm(\lambda_1)-  M^\pm(\lambda_2)= O\big( |\lambda_1-\lambda_2|^{\frac12+\gamma} |\lambda_2|^{\frac12-\gamma-}\big).
	\ee 
	\item  If $\beta>3$, then
	\be\label{eq:Mprime}
		\partial_\lambda M^\pm(\lambda)=O(\lambda^{0-}).
	\ee
	\item If $\beta>3+2\gamma$ for some $0\leq \gamma\leq 1$, then for $0<|\lambda_1|\leq |\lambda_2|\les 1$ , we have 
	\be\label{eq:Mlip1alpha}
	  \partial_\lambda M^\pm(\lambda_1)-  \partial_\lambda M^\pm(\lambda_2)= O\big( |\lambda_1-\lambda_2|^{ \gamma} |\lambda_1|^{ -\gamma- } \big).
	\ee 
\end{enumerate}
		In all statements above the error terms are understood in the   Hilbert-Schmidt norm.
\end{lemma}

We note that the uniform $L^1\to L^\infty$ bound requires only the bounds \eqref{eq:Mpntwise} and \eqref{eq:Mlip12alpha} with $\gamma=0$, hence only requiring that the entries of $V$ satisfy $|V(x)|\les \la x\ra^{-2-}$.

\begin{proof}
Using \eqref{eqn:M def}, \eqref{eqn:T def}, and  \eqref{eq:R0simple},   we have
$$
\big[M^\pm (\lambda)-T\big](x,y)= v(x)\big(\mR_0^\pm(\lambda)-\mG_{0,0} \big)(x,y)v^*(y) 
= O\Big(|\lambda|^{1-}  \,\frac{ 1+|x-y|^{0-}}{ \la x\ra^{1+} \la y\ra^{1+}}\Big).
$$
This yields \eqref{eq:Mpntwise}. To obtain \eqref{eq:Mlip12alpha}, we use \eqref{R0lip12alpha}: 
\begin{multline*}
\big[M^\pm(\lambda_1)-  M^\pm(\lambda_2)\big](x,y)=v(x)\big(\mR_0^\pm(\lambda_1)-\mR_0^\pm(\lambda_2) \big)(x,y)v^*(y)  \\
= O\Big(|\lambda_1-\lambda_2|^{\frac12+\gamma} |\lambda_2|^{\frac12-\gamma-} \frac{|x-y|^{\gamma}+|x-y|^{0-}}{\la x\ra^{1+\gamma+}  \la y\ra^{1+\gamma+}}
\Big)\\=O\Big(|\lambda_1-\lambda_2|^{\frac12+\gamma} |\lambda_2|^{\frac12-\gamma-} \frac{1+|x-y|^{0-}}{\la x\ra^{1 +}  \la y\ra^{1 +}}
\Big).
\end{multline*}
This yields \eqref{eq:Mlip12alpha}. Similarly, writing 
$$
\partial_\lambda M^\pm(\lambda)(x,y)= v(x) \partial_\lambda \mR_0^\pm(\lambda)(x,y)v^*(y),  
$$
we note that \eqref{eq:Mprime} follows from \eqref{R01}, and 
\eqref{eq:Mlip1alpha} from \eqref{R0lip1alpha}. 
\end{proof}

The following lemma establishes analogous bounds for $(M^\pm(\lambda))^{-1}$ when zero is regular. 
\begin{lemma}\label{lem:Minverse_reg}
Assume that $|V(x)|\les \la x\ra^{-\beta}$ and that zero is a regular point of the spectrum.
\\
If $\beta>2$, then  $M^\pm(\lambda)$ is invertible with a uniformly bounded inverse provided that $0<|\lambda|\ll 1$.
\begin{enumerate}[i)]
	\item If $\beta>2+2\gamma$ for some $0\leq \gamma<\frac12$, then for $0<|\lambda_1|\leq |\lambda_2|\ll 1$, we have 
	\be\label{eq:Minvlip12alpha}
	  (M^\pm(\lambda_1))^{-1}-  (M^\pm(\lambda_2))^{-1}= O\big( |\lambda_1-\lambda_2|^{\frac12+\gamma} |\lambda_2|^{\frac12-\gamma-}\big).
	\ee 
	\item  If $\beta>3$, then
	\be\label{eq:Minvprime}
		\partial_\lambda (M^\pm(\lambda))^{-1}=O(\lambda^{0-}).
	\ee
	\item If $\beta>3+2\gamma$ for some $0\leq \gamma\leq 1$, then for $0<|\lambda_1|\leq |\lambda_2|\ll 1$ , we have 
	\be\label{eq:Minvlip1alpha}
	  \partial_\lambda (M^\pm(\lambda_1))^{-1}-  \partial_\lambda (M^\pm(\lambda_2))^{-1}= O\big( |\lambda_1-\lambda_2|^{ \gamma} |\lambda_1|^{ -\gamma- }  \big).
	\ee 
\end{enumerate}
	In all statements above the error terms are understood as absolutely bounded operators.  
\end{lemma}
\begin{proof}
When zero is regular, the operator $T$ is invertible with an absolutely bounded inverse. 
Therefore, by Lemma~\ref{lem:Mexp1}, $M^\pm(\lambda)$ is invertible with a uniformly bounded inverse provided that $0<|\lambda|\ll 1$ and $|V(x)|\les \la x\ra^{-2-}$.  

Using resolvent identity, the boundedness of $(M^\pm)^{-1}$ and \eqref{eq:Mlip12alpha} we obtain \eqref{eq:Minvlip12alpha}:
\begin{multline*}
(M^\pm(\lambda_1))^{-1}-  (M^\pm(\lambda_2))^{-1}= (M^\pm(\lambda_2))^{-1} \big[M^\pm(\lambda_2)-  M^\pm(\lambda_1)\big] (M^\pm(\lambda_1))^{-1}\\ =O\big( |\lambda_1-\lambda_2|^{\frac12+\gamma} |\lambda_2|^{\frac12-\gamma-}\big).
\end{multline*}
To obtain \eqref{eq:Minvprime}, we use \eqref{eq:Mprime} and the identity
$$
\partial_\lambda (M^\pm(\lambda))^{-1}= - (M^\pm(\lambda))^{-1} \big(\partial_\lambda  M^\pm(\lambda)\big) (M^\pm(\lambda))^{-1}.
$$ 
Finally, \eqref{eq:Minvlip1alpha} follows from \eqref{eq:Minvlip12alpha}, \eqref{eq:Mprime} and \eqref{eq:Minvprime}  after writing 
\begin{multline*}
 \partial_\lambda (M^\pm(\lambda_1))^{-1}-  \partial_\lambda (M^\pm(\lambda_2))^{-1}= \big[ (M^\pm(\lambda_2))^{-1} - (M^\pm(\lambda_1))^{-1}\big] \big(\partial_\lambda  M^\pm(\lambda_2)\big) (M^\pm(\lambda_2))^{-1} \\
+  (M^\pm(\lambda_1))^{-1} \big[\partial_\lambda  \big(M^\pm(\lambda_2)\big)-\big(\partial_\lambda  M^\pm(\lambda_1)\big)\big] (M^\pm(\lambda_2))^{-1}\\
+(M^\pm(\lambda_1))^{-1} \big(\partial_\lambda  M^\pm(\lambda_1)\big) \big[(M^\pm(\lambda_2))^{-1}
- (M^\pm(\lambda_1))^{-1} \big].
\end{multline*}
\end{proof}

We are now ready to prove the small energy assertions of Theorem~\ref{thm:main} when zero is regular by studying  the small energy portion of the Stone's formula, \eqref{Stone},
$$
 \int_{-\infty}^{\infty}  e^{-it\lambda} \chi(\lambda)  [\mR_V^+ -\mR_V^-](\lambda)(x,y)\, d\lambda.
$$
In particular, we will prove the following family of bounds, which includes the uniform bound when $\gamma=0$.
\begin{prop}\label{prop:disp est reg}
	Fix $0\leq \gamma<\frac32$ and assume that $|V(x)|\les \la x\ra^{-2-2\gamma-}$.
	If zero is regular, then we have the bound
	\begin{align}
	\left| \int_{-\infty}^{\infty}  e^{-it\lambda} \chi(\lambda)  [\mR_V^+ -\mR_V^-](\lambda)(x,y)\, d\lambda  \right| 
	&\les \la x\ra^{\gamma} \la y \ra^{\gamma} \la t \ra^{-\f12-\gamma}.
	\end{align} 
\end{prop}

In \cite{egd}, the authors studied the solution operator as an operator $\mathcal H^1\to BMO$ because the operator  $\mG_{0,0}$ is not bounded from $L^1 \rightarrow L^2$ or from $L^2 \rightarrow L^{\infty}$. Simple use of iterated resolvent identity was not enough to deal with this problem in the massive case since one relies on the orthogonality properties of the most singular terms in the expansion of the operator  $M^\pm(\lambda) ^{-1} =\big( U + v \mathcal{R}_0^\pm (\lambda)  v^*\big)^{-1}$ to get uniform estimates in $x,y$. 
In \cite{EGT2d}, this problem was  overcome  by selectively using the iterated resolvent identity for $M^{\pm}(\lambda) ^{-1}$ only for certain terms arising in the expansion.   

Since we don't rely on orthogonality arguments here, we need only use the iterated symmetric resolvent identity:
\be 
\label{Rexpnew}
\mR_V^\pm= \mR_0^\pm- \mR_0^\pm V \mR_0^\pm+ \mR_0^\pm V \mR_0^\pm V  \mR_0^\pm-\mR_0^\pm V \mR_0^\pm v^*M_{\pm}^{-1}  v \mR^{\pm}_0  V  \mR_0^\pm.
\ee
We consider the contribution of the first three summands in \eqref{Rexpnew} to the Stone's formula.
\begin{lemma}\label{lem:bornlow2} Let $\Gamma^{\pm}= \mR_0^\pm-\mathcal{R}_0^\pm V\mathcal{R}_0^\pm + \mathcal{R}_0^\pm V \mR_0^\pm V   \mathcal{R}^{\pm}_0$. Then provided that $|V(x)|\les \la x\ra^{-2- 2\gamma-}$ for some $0\leq \gamma<\frac32$, then we have the bound
	$$
  \Big|\int_\R e^{-it\lambda} \chi(\lambda)[\Gamma^+-\Gamma^-](\lambda)(x,y) d\lambda \Big| 
\les \la x\ra^{\gamma} \la y \ra^{\gamma} \la t \ra^{-\f12-\gamma}. $$
\end{lemma}
\begin{proof} The contribution of the first term is the free evolution which was dealt with above in Theorem~\ref{thm:free}. We note the following useful algebraic identity
\begin{align}\label{alg fact}
	\prod_{k=0}^MA_k^+-\prod_{k=0}^M A_k^-
	=\sum_{\ell=0}^M \bigg(\prod_{k=0}^{\ell-1}A_k^-\bigg)
	\big(A_\ell^+-A_\ell^-\big)\bigg(
	\prod_{k=\ell+1}^M A_k^+\bigg),
\end{align}
It suffices to consider the contribution of the following to the integral 
$$
	\widetilde\Gamma:=  \mu_0 V\mathcal{R}_0^++\mu_0 V\mathcal{R}_0^+V\mathcal{R}_0^++  \mathcal{R}_0^- V\mu_0V\mathcal{R}_0^+,
$$
where $\mu_0(\lambda)=\chi(\lambda) (\mathcal{R}_0^+(\lambda)- \mathcal{R}_0^-(\lambda)).$ 
The remaining terms have similar structure with differences $\mu_0$ on the right instead of the left.
 
Using the bounds \eqref{R0pm0}   and  \eqref{eq:R0simple}, and noting Lemma~\ref{lem:EGcor},
 we see that the kernel of $\widetilde\Gamma$ is bounded in $\lambda, x, y$ and it is supported in $|\lambda|\les 1$.  Therefore, we restrict our attention to the case $|t|>1$.

We start with the case $0\leq \gamma<\frac12$.  Using the Lipschitz bounds 
\eqref{mu0C12alpha}, \eqref{R0lip12alpha}, and the pointwise bounds \eqref{R0pm0}, \eqref{eq:R0simple}, Lemmas~\ref{lem:EGcor} and \ref{lem:spatial int1alpha}
we see that for $|\lambda_j|\les 1$, $j=1,2$, 
$$
|\widetilde\Gamma(\lambda_1)-\widetilde\Gamma(\lambda_2)|\les |\lambda_1-\lambda_2|^{\frac12+\gamma}\la x\ra^{\gamma}\la y\ra^{\gamma}.
$$
Therefore, as in \eqref{lip12trick}, we have 
\begin{multline*}
\int_\R e^{-it\lambda}  \widetilde\Gamma(\lambda)(x,y) d\lambda  =\frac12 
 \int_{|\lambda|\les 1} e^{-it\lambda} \big[ \widetilde\Gamma(\lambda)(x,y)-   \widetilde
 \Gamma(\lambda-\frac\pi{t})(x,y) \big] d\lambda \\
 = O(|t|^{-\frac12-\gamma})\la x\ra^{\gamma}\la y\ra^{\gamma}. 
\end{multline*}
The case $\frac12\leq \gamma <\frac32$ is similar after an integration by parts.  That is, we need to bound
$$
	\int_\R e^{-it\lambda}  \widetilde\Gamma(\lambda)(x,y) d\lambda =\frac{1}{it} \int_\R e^{-it\lambda}  \partial_\lambda \widetilde\Gamma(\lambda)(x,y) d\lambda 
$$
To do this, we need Lipschitz bounds on $\partial_\lambda \widetilde \Gamma$.  Writing
$$
	\partial_\lambda \widetilde\Gamma= \partial_\lambda \big( \mu_0 V\mathcal{R}_0^+\big) + \partial_\lambda \big(\mu_0 V\mathcal{R}_0^+V\mathcal{R}_0^+\big)
	+\partial_\lambda \big(  \mathcal{R}_0^- V\mu_0V\mathcal{R}_0^+\big):=
	\Gamma_1+\Gamma_2+\Gamma_3,
$$
we seek to bound $\Gamma_j(\lambda_1)-\Gamma_j(\lambda_2)$ for $j=1,2,3$.
We consider $\Gamma_1$, the others are similar. Note that
\begin{multline}\label{eq:Gamma1 diff}
	\Gamma_1(\lambda_1)-\Gamma_1(\lambda_2)=[\partial_\lambda\mu_0 (\lambda_1)-\partial_\lambda\mu_0(\lambda_2)]V\mathcal{R}_0^+(\lambda_1)+  \partial_\lambda\mu_0(\lambda_2)V[\mathcal{R}_0^+(\lambda_1)-\mathcal{R}_0^+(\lambda_2)]\\ +[\mu_0 (\lambda_1)- \mu_0(\lambda_2)]V\partial_\lambda\mathcal{R}_0^+(\lambda_1)+\mu_0(\lambda_2)V[\partial_\lambda\mathcal{R}_0^+(\lambda_1)-\partial_\lambda\mathcal{R}_0^+(\lambda_2)].
\end{multline}

Let $\gamma_0\in[0,1)$ be such that $\gamma=\gamma_0+\frac{1}{2}$ and using \eqref{mu0C32}  and  \eqref{R0lip1alpha}, (for consistency, we take $|\lambda_1|\leq |\lambda_2|$)
\begin{align*}
|\partial_\lambda \mu_0(\lambda_1)(x,y)-\partial_\lambda \mu_0(\lambda_2)(x,y)| \les  & |\lambda_1-\lambda_2|^{\gamma_0} |x-y|^{\gamma_0} (1+|\lambda_2||x-y|)^{\frac12} \\
&\les |\lambda_1-\lambda_2|^{\gamma_0} \la x\ra^\gamma \la y\ra^{\gamma},  \\
\big|\partial_{\lambda}\mR_0^\pm(\lambda_1)(x,y)-\partial_\lambda\mR_0^\pm(\lambda_2)(x,y)\big| \les &|\lambda_1-\lambda_2|^{\gamma_0} |\lambda_1|^{-\gamma_0-} \big(  |x-y|^{0-} +|x-y|^{\frac12+\gamma_0}\big)\\
&\les |\lambda_1-\lambda_2|^{\gamma_0} |\lambda_1|^{-\gamma_0-} \la x\ra^\gamma \la y\ra^{\gamma} (1+|x-y|^{0-}).
\end{align*}
In addition using \eqref{mu0C12alpha} with $\gamma=\frac12$ we have
$$
 |\mu_0(\lambda_1)(x,y)-\mu_0(\lambda_2)(x,y)| \les |\lambda_1-\lambda_2|   \la x-y\ra^{\frac12}\les  |\lambda_1-\lambda_2|^{\gamma_0}   \la x\ra^\frac12 \la y\ra^{\frac12}. 
$$
Where the last bound follows since $|\lambda_1-\lambda_2|<1$ and $\gamma_0<1$.
Similarly, using \eqref{R0lip12alpha} with $\gamma=\frac12-$ we obtain
$$
\big|\mR_0^\pm(\lambda_1)(x,y)-\mR_0^\pm(\lambda_2)(x,y)\big| \les  |\lambda_1-\lambda_2|^{\gamma_0}   \la x\ra^\frac12 \la y\ra^{\frac12}  (1+|x-y|^{0-}),
$$
Finally   by   \eqref{R0pm0} and  \eqref{eq:R0simple}, we have
$$
|\mu_0(\lambda)(x,y)|\les 1,\,\,\,\,|\mR_0^\pm(\lambda )(x,y)|\les (1+|x-y|^{-1}).
$$
Putting this all together and using Lemma~\ref{lem:EGcor}, we see that
\begin{multline*}
|\Gamma_1(\lambda_1)-\Gamma_1(\lambda_2)|\les |\lambda_1-\lambda_2|^{\gamma_0} |\lambda_1|^{-1+} \la x\ra^{\gamma}\la y\ra^{\gamma} \int_{\R^2}  \la y_1\ra^{-2-} (1+|y-y_1|^{-1}) dy_1 \\ \les |\lambda_1-\lambda_2|^{\gamma_0} |\lambda_1|^{-1+} \la x\ra^{\gamma}\la y\ra^{\gamma}.
\end{multline*}
Similarly, using Lemmas~\ref{lem:spatial int1alpha} and  \ref{lem:spa_int1}, we see that  $\Gamma_2$ and $\Gamma_3$ satisfy the same estimate. Thus
\begin{multline*}
	\int_\R e^{-it\lambda}  \widetilde\Gamma(\lambda)(x,y) d\lambda =\frac{1}{it} \int_\R e^{-it\lambda}  \partial_\lambda \widetilde\Gamma(\lambda)(x,y) d\lambda    
	\\=\frac{1}{2it} 
	\int_{|\lambda|\les 1} e^{-it\lambda} \big[\partial_\lambda  \widetilde\Gamma(\lambda)(x,y)-  \partial_\lambda  \widetilde
	\Gamma(\lambda-\frac\pi{t})(x,y) \big] d\lambda \\  =O(|t|^{-1-\gamma_0})\la x\ra^{\gamma}\la y\ra^{\gamma}
	= O(|t|^{-\frac12-\gamma})\la x\ra^{\gamma}\la y\ra^{\gamma}. 
\end{multline*}

\end{proof}

The lemma below takes care of the contribution of $M^{-1}$ term for $0\leq \gamma<\frac12$.   In contrast to the massive case \cite{egd,EGT2d} or Schr\"odinger \cite{eg3}, for the massless Dirac bound, the argument employed here does not require any cancellation between the `+' and `-' terms in the Stone's formula, \eqref{Stone}.
\begin{lemma}\label{prop:gen} Fix $0\leq \gamma<\frac12$. Assume that $|V(x)|\les \la x\ra^{-2-2\gamma-}$. Let $T(\lambda)$ be an absolutely bounded operator  satisfying
(for $|\lambda|,|\lambda_1|,|\lambda_2|\les 1$ with $|\lambda_1|\leq |\lambda_2|$) 
$$
\big\| |T(\lambda ) |\big\|_{L^2\to L^2} \les |\lambda|^{-1+},
$$
$$
\big\| |T(\lambda_1)-T(\lambda_2)|\big\|_{L^2\to L^2}\les |\lambda_1|^{-1+} |\lambda_1-\lambda_2|^{\frac12+\gamma}.
$$
 Then
$$
\Big|\int_\R e^{-it\lambda} \chi(\lambda)\big[\mR_0^\pm  V \mR_0^\pm   v^*T   v \mR^{\pm}_0   V  \mR_0^\pm\big](\lambda)(x,y) d\lambda \Big|\les \la t\ra^{-\frac12-\gamma} \la x\ra^{\gamma} \la y \ra^\gamma.
$$  
\end{lemma}
Note that the hypothesis  is satisfied by the mean value theorem if  $T(\lambda)=\widetilde O_1(\lambda^{-\frac12+})$ as  an absolutely bounded operator. Also note that when zero is regular $M^{-1}$ satisfies the hypothesis provided that $|V(x)|\les \la x\ra^{-2-2\gamma-}$, see Lemma~\ref{lem:Minverse_reg}. 
\begin{proof}
 Dropping $\pm$ signs, let $\widetilde R:= v \mR_0  V  \mR_0$.
Using the support of $\chi(\lambda)$ as well as the bounds \eqref{eq:R0simple} and \eqref{R0lip12alpha} for the free resolvent and the integral estimates in Lemmas~\ref{lem:spa_int1} and \ref{lem:spatial int1alpha}
we have (provided that $|V(x)|\les \la x\ra^{-2-2\gamma-}$, $0\leq \gamma<\tfrac12$) 
\be\label{tildeR}
|\widetilde R(\lambda)(y_1,y)| \les   (1+|y_1-y|^{0-})\la y_1\ra^{-1-}
\ee
\be\label{tildeRlip}
|\widetilde R(\lambda_1)(y_1,y)-\widetilde R(\lambda_2)(y_1,y)| \les   |\lambda_1-\lambda_2|^{\frac12+\gamma} |\lambda_2|^{\frac12-\gamma-} \la y\ra^{\gamma} \la y_1\ra^{-1-}.
\ee
 Note that \eqref{tildeR} and Lemma~\ref{lem:EGcor} imply that $L^2_{y_1}$ norm of $\widetilde R(\lambda)(y_1,y)$ is bounded uniformly in  $y$ and $\lambda$, while \eqref{tildeRlip} implies that the $L^2_{y_1}$ norm of $\widetilde R(\lambda_1)(y_1,y)-\widetilde R(\lambda_2)(y_1,y)$ is bounded by  $\la y\ra^{\gamma} |\lambda_1-\lambda_2|^{\frac12+\gamma} $.

Using these bounds and the hypothesis for $T$, using \eqref{alg fact} we see that (with $\Gamma:=\mR_0^\pm  V \mR_0^\pm   v^*T   v \mR^{\pm}_0   V  \mR_0^\pm$)
$$
|\Gamma(\lambda_1)-\Gamma(\lambda_2)|\les \la x\ra^{\gamma}\la y\ra^{\gamma} |\lambda_1-\lambda_2|^{\frac12+\gamma} |\lambda_1|^{-1+},\,\,|\lambda_j|\ll 1, j=1,2.
$$
We use \eqref{eq:R0simple} and \eqref{R0lip12alpha} for the free resolvent terms.  
Therefore, by applying the Lipschitz argument as in \eqref{lip12trick} and the proof of Lemma~\ref{lem:bornlow2}, we bound the integral  by
$$
 \la t\ra^{-\frac12-\gamma} \la x\ra^{\gamma}\la y\ra^{\gamma} \int_{-1}^1 \big(\min(|\lambda |,|\lambda-\tfrac\pi{t}|)\big)^{-1+} d\lambda\les \la t\ra^{-\frac12-\gamma} \la x\ra^{\gamma}\la y\ra^{\gamma}.
$$   
\end{proof}
For $\frac12\leq \gamma<\frac32$, we have the following lemma which we state only for $M^{-1}$. We dropped $\pm$ signs since we won't rely on any cancellation between $\pm$ terms. 
\begin{lemma} \label{prop:gen1} Fix $\frac12\leq \gamma<\frac32$. Assume that $|V(x)|\les \la x\ra^{-2-2\gamma -}$.  
 Then
$$
\Big|\int_\R e^{-it\lambda} \chi(\lambda)\big[\mR_0   V \mR_0   v^*M^{-1}  v \mR_0   V  \mR_0\big](\lambda)(x,y) d\lambda \Big|\les \la t\ra^{-\frac12-\gamma} \la x\ra^{\gamma} \la y \ra^\gamma.
$$  
\end{lemma}  
\begin{proof}
We only need consider the case $|t|>1$.  
Let $\gamma_0=\gamma-\frac12$. After an integration by parts, and ignoring the case when the derivative hits the cutoff $\chi$, it suffices to prove that
$$
\Big|\int_\R e^{-it\lambda} \chi(\lambda) \partial_\lambda\big[\mR_0   V \mR_0   v^*M^{-1}   v \mR_0   V  \mR_0 \big](\lambda)(x,y) d\lambda \Big|\les |t|^{-\gamma_0} \la x\ra^{\gamma} \la y \ra^\gamma.
$$  
Let $\widetilde R:= v \mR_0  V  \mR_0$ as in the proof of Lemma~\ref{prop:gen}.
Since $|V(x)|\les \la x\ra^{-2-2\gamma-}$,   the bound \eqref{tildeR} is valid. Using \eqref{tildeRlip} with $\gamma=\gamma_0-\tfrac12$ for $\gamma_0\in (\tfrac12,1) $ and with $\gamma=0$ for $\gamma_0\in (0,\tfrac12]$, we have 
$$
|\widetilde R(\lambda_1)(y_1,y)-\widetilde R(\lambda_2)(y_1,y)| \les   |\lambda_1-\lambda_2|^{\gamma_0} \big(1+ \la y\ra^{\gamma_0-\frac12}\big) \la y_1\ra^{-1-}
$$ 
Using \eqref{eq:R0simple}, \eqref{R01}, and integral estimate Lemma~\ref{lem:spatial int1alpha} (with $\gamma=\tfrac12$), we have
\be\label{tildeRprime}
|\partial_\lambda \widetilde R(\lambda)(y_1,y)| \les  |\lambda|^{0-}  \la y\ra^{\frac12} \la y_1\ra^{-1-}.
\ee
Finally we need a Lipschitz bound for $\partial_\lambda \widetilde R$. First note that  
using \eqref{R0lip12alpha} with $\gamma=\gamma_0-\tfrac12$ for $\gamma_0\in (\tfrac12,1) $ and with $\gamma=0$ for $\gamma_0\in (0,\tfrac12]$, we have 
$$
\big|\mR_0^\pm(\lambda_1)(x,y)-\mR_0^\pm(\lambda_2)(x,y)\big|\les \\ |\lambda_1-\lambda_2|^{ \gamma_0}   (1+\la x-y\ra^{\gamma_0-\frac12}+|x-y|^{0-}).
$$ 
Moreover, recalling \eqref{R0lip1alpha}, and taking $|\lambda_1|\leq |\lambda_2|$ as usual, we have   
\begin{align}\label{R0lip1alpha new}
\big|\partial_{\lambda}\mR_0^\pm(\lambda_1)(x,y)-\partial_\lambda\mR_0^\pm(\lambda_2)(x,y)\big|  \les  |\lambda_1-\lambda_2|^{\gamma_0}  |\lambda_1|^{-\gamma_0-}\big(  |x-y| ^{0-} + |x-y|^{\gamma}\big) .
\end{align}
Using these, \eqref{eq:R0simple}, and \eqref{R01}, we obtain 
\begin{multline*}
|\partial_\lambda \widetilde R(\lambda_1)(y_1,y)-\partial_\lambda \widetilde R(\lambda_2)(y_1,y)| \les |\lambda_1-\lambda_2|^{\gamma_0} |\lambda_1|^{-\gamma_0-} \\\times \int_{\R^2} \la y_1\ra^{-1-\gamma-} \big(|y_1-y_2|^{-1}+|y_1-y_2|^\gamma\big) \la y_2\ra^{-2-2\gamma-} \big(|y_2-y |^{-1}+|y_2-y |^\gamma\big) dy_2 
\\ \les |\lambda_1-\lambda_2|^{\gamma_0} |\lambda_1|^{-\gamma_0-}  \la y\ra^{\gamma}(1+|y_1-y|^{0-})\la y_1\ra^{-1-}.
\end{multline*}
Where the spatial integral is bounded by noting that $|x-y|^\gamma \leq \la x\ra^{\gamma} \la y \ra^\gamma$ and using Lemma~\ref{lem:spatial int1alpha}.
Using these pointwise bounds we have
$$
\|\widetilde R(\lambda)(y_1,y)\|_{L^2_{y_1}}\les 1,\,\,\,\,\,\|\partial_\lambda\widetilde R(\lambda)(y_1,y)\|_{L^2_{y_1}}\les |\lambda|^{0-}\la y\ra^{\gamma},
$$
$$
\|\widetilde R(\lambda_1)(y_1,y)-\widetilde R(\lambda_2)(y_1,y)\|_{L^2_{y_1}}\les |\lambda_1-\lambda_2|^{\gamma_0} \la y\ra^{\gamma},
$$
$$
\|\partial_\lambda \widetilde R(\lambda_1)(y_1,y)-\partial_\lambda \widetilde R(\lambda_2)(y_1,y)\|_{L^2_{y_1}}\les |\lambda_1-\lambda_2|^{\gamma_0}  |\lambda_1|^{-\gamma_0-} \la y\ra^{\gamma}.
$$
Finally note that by Lemma~\ref{lem:Minverse_reg}, $M^{-1}$ satisfies similar bounds (without $x,y$ dependence) as an absolutely bounded operator.
Therefore, letting $\Gamma= \partial_\lambda \big[\mR_0   V \mR_0   v^*M^{-1}   v \mR_0   V  \mR_0 \big]$, we see that
$$
|\Gamma(\lambda_1)(x,y)-\Gamma(\lambda_2)(x,y)|\les |\lambda_1-\lambda_2|^{\gamma_0}  |\lambda_1|^{-\gamma_0-}.
$$
This finishes the proof using the Lipschitz argument as in \eqref{lip12trick} and the proof of Lemmas~\ref{lem:bornlow2} and \ref{prop:gen}.
\end{proof}
 
We now prove Proposition~\ref{prop:disp est reg}.

\begin{proof}[Proof of Proposition~\ref{prop:disp est reg}]
Using the expansion \eqref{Rexpnew}, we see that the first terms are controlled by Lemma~\ref{lem:bornlow2}.  Then it remains only to control the tail of the Born series, with the operators $M^{\pm}(\lambda)^{-1}$. 
	By the expansion for $M^{\pm}(\lambda)^{-1}$ 
	in Lemma~\ref{lem:Minverse_reg}, we see that Lemma~\ref{prop:gen} suffices to establish the desired bound for $0\leq \gamma<\tfrac12$. The case $\tfrac12\leq\gamma<\tfrac32$ is established in Lemma~\ref{prop:gen1}.
\end{proof}

\section{Small energy resolvent expansion when zero is not regular} \label{sec:resolv notfree}

We now consider the case when zero is not a regular point of the spectrum.  We first provide the necessary expansions to develop the spectral measure when there are eigenvalues and/or resonances at zero energy, then establish the dispersive estimates.
We re-emphasize here that this is the first result, to our knowledge, in which the contribution of a `p-wave' resonance is controlled in a finite-rank term.  Previous results in the Schr\"odinger (or wave equation) context, \cite{Mur,eg2,Gwave}, have not achieved this.  Even in the weighted $L^2$ setting, \cite{Mur}, any finite rank pieces had an error whose decay was only logarithmically better.  This argument can be modified to apply to the Schr\"odinger evolution as well.

With $S_1$ being the Riesz projection onto the kernel of $T$, 
define $(T+S_1)^{-1}:=T_1$.  One can see that $S_1T_1=T_1S_1=S_1$.  Then, we have the following
variations of Lemma~\ref{lem:Mexp1} and  Lemma~\ref{lem:Minverse_reg}.  

\begin{lemma}\label{lem:Mexp2}
Assume that $|V(x)|\les \la x\ra^{-\beta}$. 
 If $\beta>2+2k$ for some $0<k<1$,  then
$$
M^\pm(\lambda)=T+\lambda g^\pm(\lambda) v\mG_{1,1}v^*+\lambda v\mG_{1,0}v^*+E_1^\pm(\lambda),
$$
where 
$$\|E_1^\pm(\lambda)\|_{HS}\les |\lambda|^{1+k}.
$$
Moreover, for fixed $0\leq \gamma<\tfrac12$ and $\tfrac12\leq k<1$,  if $\beta>2+2k$, then (for $|\lambda_1|\leq|\lambda_2|\les 1$)
$$
\|E_1^\pm(\lambda_1)-E_1^\pm(\lambda_2)\|_{HS}\les |\lambda_1-\lambda_2|^{\frac12+\gamma} |\lambda_2|^{\frac12-\gamma+k}.
$$   
\end{lemma}
\begin{proof} The lemma immediately follows from the bounds in Lemma~\ref{lem:R0E0} noting that
   $E_1^\pm(\lambda)=vE_0^\pm(\lambda)v^*$.
\end{proof}

\begin{lemma}\label{lem:M+S1inv}
Assume that $|V(x)|\les \la x\ra^{-\beta}$ and that zero is not a regular point of the spectrum.
\begin{enumerate}[i)]

	\item If $\beta>2+2k$ for some $0<k<1$,  then  $M^\pm(\lambda)+S_1$ is invertible with a uniformly bounded inverse provided that $0<|\lambda|\ll 1$, and we have 
\be\label{eq:M+S1_E2exp}
	(M^\pm(\lambda)+S_1)^{-1}=T_1-\lambda g^\pm(\lambda) T_1v\mG_{1,1}v^*T_1-\lambda T_1v\mG_{1,0}v^*T_1+E_2^\pm(\lambda),
	\ee
	where 
	$$E_2^\pm(\lambda)=O(|\lambda|^{1+k}).$$  
	\item If $\beta>2+2\gamma$ for some $0\leq \gamma<\frac12$, then for $0<|\lambda_1|\leq |\lambda_2|\ll 1$, we have 
	\be\label{eq:M+S1invlip12alpha}
	  (M^\pm(\lambda_1)+S_1)^{-1}-  (M^\pm(\lambda_2)+S_1)^{-1}= O\big( |\lambda_1-\lambda_2|^{\frac12+\gamma} |\lambda_2|^{\frac12-\gamma-}\big).
	\ee 
Moreover, 	for fixed $0\leq \gamma<\tfrac12$ and $\tfrac12\leq k<1$,  if $\beta>2+2k$, then (for $|\lambda_1|\leq|\lambda_2|\ll 1$)  
\be\label{eq:E2lip12alpha}
	  E_2^\pm(\lambda_1)-   E_2^\pm(\lambda_2)= O\big( |\lambda_1-\lambda_2|^{\frac12+\gamma} |\lambda_2|^{\frac12-\gamma+k}\big).
	\ee 
\end{enumerate}
	All bounds above are understood in the sense of absolutely bounded operators.
\end{lemma}
\begin{proof}
The first assertion follows from the invertibility of $T+S_1$, \eqref{eq:Mpntwise} 
and a Neumann Series computation. Recalling that $T_1=(T+S_1)^{-1}$, the expansion \eqref{eq:M+S1_E2exp}  follows from Lemma~\ref{lem:Mexp2} noting that 
\begin{multline*}
(M^\pm(\lambda)+S_1)^{-1}=\big[T+S_1+\lambda g^\pm(\lambda) v\mG_{1,1}v^*+\lambda v\mG_{1,0}v^*+E_1^\pm(\lambda)\big]^{-1}\\
=\big[I+\lambda g^\pm(\lambda) T_1v\mG_{1,1}v^*+\lambda vT_1\mG_{1,0}v^*+T_1E_1^\pm(\lambda)\big]^{-1}T_1\\
=T_1-\lambda g^\pm(\lambda) T_1v\mG_{1,1}v^*T_1-\lambda vT_1\mG_{1,0}v^*T_1-T_1E_1^\pm(\lambda) T_1+\sum_{j=2}^\infty (-1)^j \Gamma^j T_1,
\end{multline*}
where $\Gamma=\lambda g^\pm(\lambda) T_1v\mG_{1,1}v^*+\lambda vT_1\mG_{1,0}v^*+T_1E_1^\pm(\lambda) =O(|\lambda|^{1-})$.
Therefore (since $k<1$),  
$$E_2^\pm(\lambda)=-T_1E_1^\pm(\lambda) T_1+\sum_{j=2}^\infty (-1)^j \Gamma^j T_1=O(|\lambda|^{1+k}).
$$
The proof of \eqref{eq:M+S1invlip12alpha} is identical to the proof of \eqref{eq:Minvlip12alpha}. Finally \eqref{eq:E2lip12alpha} follows from 
the Lipschitz bound for $E_1^\pm$ in Lemma~\ref{lem:Mexp2}, the bound $\Gamma  =O(|\lambda|^{1-}) $,  and by noting that the first two terms in the definition of $\Gamma$ satisfies the Lipschitz bound 
$$ |\lambda_1-\lambda_2|^{\frac12+\gamma} |\lambda_2|^{\frac12-\gamma-}.$$ 
\end{proof}

To invert $M^\pm(\lambda)=U+v\mR_0^\pm(\lambda^2)v$, for small $\lambda$, we use the
following  lemma (see Lemma 2.1  in \cite{JN}) repeatedly.  

\begin{lemma}\label{JNlemma}
Let $M$ be a closed operator on a Hilbert space $\mathcal{H}$ and $S$ a projection. Suppose $M+S$ has a bounded
inverse. Then $M$ has a bounded inverse if and only if
$$
B:=S-S(M+S)^{-1}S
$$
has a bounded inverse in $S\mathcal{H}$, and in this case
$$
M^{-1}=(M+S)^{-1}+(M+S)^{-1}SB^{-1}S(M+S)^{-1}.
$$
\end{lemma}

We   apply this lemma with $M=M^\pm(\lambda)$ and $S=S_1$. The fact that $M^{\pm}(\lambda)+S_1$
has a bounded inverse in $L^2(\mathbb R^2)$ follows from Lemma~\ref{lem:M+S1inv}. We also need to prove that   
\begin{align}\label{B defn}
  B_{\pm}=S_1-S_1(M^\pm(\lambda)+S_1)^{-1}S_1
\end{align}
has a bounded inverse in $S_1L^2(\mathbb R^2)$. We have, using \eqref{eq:M+S1_E2exp}  and the fact that $S_1T_1=S_1$,
\begin{align*}
B^\pm (\lambda)&= S_1-S_1(M^\pm (\lambda)+S_1)^{-1}S_1\\
&=S_1-S_1 \bigg[ T_1-\lambda g^\pm(\lambda) T_1v\mG_{1,1}v^*T_1-\lambda T_1v\mG_{1,0}v^*T_1+E_2^\pm(\lambda)\bigg] S_1\\
&=\lambda g^\pm (\lambda) S_1v\mG_{1,1}v^*S_1+\lambda S_1v\mG_{1,0}v^*S_1-S_1E_2^\pm(\lambda) S_1.
\end{align*}
We write:
\begin{align}\label{eqn:B defn}
	B^\pm(\lambda)&=\lambda A^\pm (\lambda)-S_1E_2^\pm(\lambda) S_1\\
	A^\pm(\lambda)&=S_1v(g^\pm(\lambda)\mG_{1,1}+\mG_{1,0}  )v^*S_1 \label{eqn:A defn}
\end{align}
The remainder of this section is devoted to inverting $A^\pm(\lambda)$ in a neighborhood of zero under different spectral assumptions.

\begin{prop}\label{lem:Ainverse}
	Assume that $|V(x)|\les \la x\ra^{-2-}$. For sufficiently small $\lambda$, the operators $A^\pm(\lambda)$ are invertible on $S_1L^2$.  Further, 
	$$
	A^\pm(\lambda)^{-1}=[S_2v\mG_{1,0}v^*S_2]^{-1}+\widetilde O_1((\log\lambda)^{-1} ),	
	$$
	as an operator on $S_1L^2$.  Morever 
	$$
	A^+(\lambda)^{-1}-A^-(\lambda)^{-1}=
	\widetilde O_1((\log\lambda)^{-2} ).
	$$
	Furthermore, if $S_1=S_2$, we have
	$$
		A^\pm(\lambda)^{-1}= [S_2v\mG_{1,0}v^*S_2]^{-1},
	$$
	which is independent of $\lambda$ and the choice of sign.
	
\end{prop}

We note that these operators are finite rank on $L^2$ since $S_1L^2$ is a finite-dimensional subspace.

\begin{proof}

We begin by writing the projection $S_1=Q\oplus S_2$ where $Q$ is orthogonal to $S_2$. We note that by Lemmas~\ref{lem:S1char2} and \ref{lem:S2class}, $Q$ corresponds to a projection onto the p-wave resonance space.  By Corollary~\ref{cor:dim}, $Q$  has rank at most two.  We first note that when $Q=0$, the statement follows \eqref{eqn:A defn} and the orthogonality property that $S_2v\mG_{1,1}=0$.  
The invertibility of the resulting operator is guaranteed by Lemma~\ref{lem:S2invert}.  The following lemma implies the proposition when $S_2=0$.

\begin{lemma}\label{lem:QAQ invert}
	
	When $Q\neq 0$, the operator $QA^\pm(\lambda)Q$ is invertible for sufficiently small $\lambda$.  Further,
	$$
		(QA^\pm(\lambda)Q)^{-1}=\widetilde O_1((\log\lambda)^{-1} ),	
	$$
	as an operator on $QL^2$.  Morever 
	$$
		(QA^+(\lambda)Q)^{-1}-(QA^-(\lambda)Q)^{-1}=
		\widetilde O_1((\log\lambda)^{-2} ).
	$$
	
\end{lemma}

\begin{proof}
 We begin by showing that $QA^\pm(\lambda)Q$ is invertible on $QL^2$.  In the case that $Q$ has rank one,  then using \eqref{eqn:A defn} we can see that $QA^\pm(\lambda)Q$ is a scalar of the form
$$
	(c_1g^\pm(\lambda)+c_2)Q, \qquad c_1\in \mathbb R\setminus \{0\}.
$$
Which, by \eqref{g def}, suffices to show our desired results.  

We now consider the case when $Q$ has rank two.
We may select an orthonormal basis for $QL^2$, $\{\phi_1, \phi_2 \}$.  We claim that $ \mG_{1,1}v^* \phi_1$ and $ \mG_{1,1}v^* \phi_2$ are linearly independent.  Assume they aren't, and let $\psi_j=-\mG_{0,0} v^* \phi_j$, $j=1,2$.  
Then for some $c$, 
$$ \psi_1-c\psi_2=\mG_{0,0}v^*(c\phi_2-\phi_1)
=\big(\mG_{0,0}-\frac{i\alpha \cdot x}{2\pi \la x\ra^2} \mG_{1,1}\big)v^*(c\phi_2-\phi_1)\in L^2$$
by the proof of Lemma~\ref{lem:S1char}.  By Lemma~\ref{lem:S1char2}, $\psi(x)=\frac{-i\alpha \cdot x}{2\pi \la x\ra^2}\mG_{1,1} v^* \phi  +\Gamma_2$ with $\Gamma_2\in L^2$. Hence $\{\phi_1,\phi_2\}$ can only span a one-dimensional subspace of $QL^2$.  This proves our claim.

We now write with respect to the  basis $\{\phi_1, \phi_2 \}$:
$$
QA^\pm(\lambda)Q=g^\pm(\lambda) \left[\begin{array}
{ll}|\mG_{1,1}v^*\phi_1|^2 & \la \mG_{1,1}v^*\phi_1, \mG_{1,1}v^*\phi_2\ra_{\mathbb C^2} \\
\overline{\la \mG_{1,1}v^*\phi_1, \mG_{1,1}v^*\phi_2\ra_{\mathbb C^2}}
& |\mG_{1,1}v^*\phi_2|^2 \end{array}
\right] + A_1,
$$
where $A_1$ is a $2\times 2$ matrix of constants given by the contributions of $\phi_i v\mG_{1,0}v^* \phi_j$. Since $ \mG_{1,1}v^* \phi_1$ and $ \mG_{1,1}v^* \phi_2$ are linearly independent, the first matrix above is invertible, and hence, for sufficiently small $\lambda$, $QA^\pm(\lambda)Q$ is invertible. Moreover the entries of its inverse are rational functions in $\log(\lambda)$, and the degree of the denominator is at least one more than the degree of the numerator. In particular, they are of the form 
$\widetilde O_1(\frac1{\log(\lambda)})$.

The final claim follows from the resolvent identity and \eqref{g def}, since $(A^+-A^-)(\lambda)$ is independent of $\lambda$.

\end{proof}

We now consider the case when both $Q,S_2\neq 0$.  We  employ the Feshbach formula, see for example Lemma~2.3 in \cite{JN}.  If $A(\lambda)=\left[\begin{array}{ll} a_{11} & a_{12}\\ a_{21} & a_{22}\end{array} \right]$, the invertibility of $A(\lambda)$ follows if both $a_{22}$ is invertible and  $a:=(a_{11}-a_{12}a_{22}^{-1}a_{21})^{-1}$ exists.  Then, we have
\begin{align}
	A(\lambda)^{-1}=\left[
	\begin{array}{ll}
		a & -aa_{12}a_{22}^{-1}\\
		-a_{22}^{-1}a_{21}a & a_{22}^{-1}a_{21} a a_{12} a_{22}^{-1}+a_{22}^{-1}
	\end{array}
	\right].
\end{align}
In our case $a_{22}=S_2v\mG_{1,0}v^*S_2$ which is invertible by Lemma~\ref{lem:S2invert}.  Moreover, 
$$
	a=\big(  QA^\pm(\lambda) Q -Qv\mG_{1,0}v^* S_2(S_2v\mG_{1,0}v^* S_2)^{-1} S_2 v\mG_{1,0}v^* Q
	\big)^{-1}
$$
exists for sufficiently small $\lambda$ since $QA^\pm(\lambda)Q$ is invertible by Lemma~\ref{lem:QAQ invert}, while the second summand is a $\lambda$ independent $2\times2$ matrix.

\end{proof}

\begin{lemma}\label{lem:Binv}
Assume that $|V(x)|\les \la x\ra^{-\beta}$ and that zero is not a regular point of the spectrum. 
If $\beta>2+2k$ for some $0<k<1$,  then for $0<|\lambda|\ll 1$, we have
$$
B_\pm(\lambda)^{-1}= \frac1\lambda  A^\pm (\lambda)^{-1}   + E_3^\pm(\lambda), 
$$
where $E_3^\pm(\lambda)= O(|\lambda|^{-1+k})$ as an absolutely bounded operator.\\
Moreover, 	for fixed $0\leq \gamma<\tfrac12$ and $\tfrac12\leq k<1$,  if $\beta>2+2k$, then (for $|\lambda_1|\leq|\lambda_2|\ll 1$)  
\be\label{eq:E3lip12alpha}
	  E_3^\pm(\lambda_1)-   E_3^\pm(\lambda_2)= O\big( |\lambda_1-\lambda_2|^{\frac12+\gamma} |\lambda_1|^{-\frac32-\gamma+k}\big).
	\ee 
\end{lemma}
\begin{proof} 
Using \eqref{eqn:B defn}, Proposition~\ref{lem:Ainverse}, and Lemma~\ref{lem:M+S1inv}, we have  
\begin{multline*}
B_\pm(\lambda)^{-1}=\frac1\lambda
\big[I-\tfrac1\lambda A^\pm (\lambda)^{-1}S_1E_2^\pm(\lambda) S_1\big]^{-1}A^\pm (\lambda)^{-1}\\ = \frac1\lambda A^\pm (\lambda)^{-1}  +\frac1\lambda \sum_{j=1}^\infty (\tfrac1\lambda A^\pm (\lambda)^{-1}S_1E_2^\pm(\lambda) S_1)^j A^\pm (\lambda)^{-1}.
\end{multline*}
The series converges since $A^\pm (\lambda)^{-1}=O(1)$ and $E_2^\pm(\lambda)=O(|\lambda|^{1+k })$ by Proposition~\ref{lem:Ainverse} and Lemma~\ref{lem:M+S1inv} respectively. Moreover, we have 
$$
E_3^\pm(\lambda)= \frac1\lambda \sum_{j=1}^\infty (\tfrac1\lambda A^\pm (\lambda)^{-1}S_1E_2^\pm(\lambda) S_1)^j A^\pm (\lambda)^{-1}= O(|\lambda|^{-1+k}).
$$
This also implies the Lipschitz bound when $|\lambda_1-\lambda_2|\gtrsim |\lambda_2|$. 
The Lipschitz bound when $|\lambda_1-\lambda_2|\ll |\lambda_2|\approx |\lambda_1|$ follows by  noting that in this case
\be\label{eq:Ainv Lip}
A^\pm (\lambda_1)^{-1}-A^\pm (\lambda_2)^{-1}=O(|\lambda_1-\lambda_2||\lambda_1|^{-1})=O(|\lambda_1-\lambda_2|^{\frac12+\gamma}| \lambda_1|^{-\frac12-\gamma}),
\ee
$$
|\lambda_2^{-2}-\lambda_1^{-2}|\les |\lambda_1-\lambda_2|^{\frac12+\gamma}| \lambda_1|^{-\frac52-\gamma},
$$
and by using the bounds in Lemma~\ref{lem:M+S1inv} for  $E_2^\pm(\lambda) $.  
\end{proof}

We are now ready to obtain a suitable  expansion for $M^\pm(\lambda)^{-1}$ when zero is not regular. Note that Proposition~\ref{lem:Ainverse} and its proof gives detailed expansions for  $A^\pm (\lambda)^{-1}$, in particular, the projection $Q$ corresponds to the contribution of p-wave resonances and  the operator $[S_2v\mG_{1,0}v^*S_2]^{-1}$ to the threshold eigenspace, see Lemma~\ref{lem:eigenproj} below. 
\begin{lemma}\label{lem:Minv p} Under the hypothesis of Lemma~\ref{lem:Binv},   for $0<|\lambda|\ll 1$, we have
$$
M^\pm(\lambda)^{-1}= \frac1\lambda S_1 A^\pm (\lambda)^{-1} S_1  + E_4^\pm(\lambda), 
$$
where $E_4^\pm(\lambda)$ satisfies the same bounds as  $E_3^\pm(\lambda)$ in Lemma~\ref{lem:Binv}.  
\end{lemma}
\begin{proof} Using Lemma~\ref{JNlemma} with $M=M^\pm(\lambda)$ and $S=S_1$, and recalling that $T_1S_1=S_1T_1=S_1$, writing $(M^\pm(\lambda)+S_1)^{-1} =[(M^\pm(\lambda)+S_1)^{-1} -T_1]+T_1$,  we have 
\begin{multline*} 
M^\pm(\lambda)^{-1}=(M^\pm(\lambda)+S_1)^{-1}+(M^\pm(\lambda)+S_1)^{-1}S_1B_\pm^{-1}S_1
(M^\pm(\lambda)+S_1)^{-1}\\
= S_1B_\pm^{-1}S_1 + (M^\pm(\lambda)+S_1)^{-1} +   [(M^\pm(\lambda)+S_1)^{-1}-T_1]S_1B_\pm^{-1}S_1
(M^\pm(\lambda)+S_1)^{-1} \\  +  S_1B_\pm^{-1}S_1[(M^\pm(\lambda)+S_1)^{-1}-T_1].
\end{multline*}
Using Lemma~\ref{lem:Binv}, we have
$$
M^\pm(\lambda)^{-1}=\frac1\lambda S_1A^\pm (\lambda)^{-1}S_1 +E_4^\pm(\lambda), 
$$
where
\begin{multline*}
E_4^\pm(\lambda)=S_1E_3^\pm(\lambda)S_1 +(M^\pm(\lambda)+S_1)^{-1} +   [(M^\pm(\lambda)+S_1)^{-1}-T_1]S_1B_\pm^{-1}S_1
(M^\pm(\lambda)+S_1)^{-1} \\  +  S_1B_\pm^{-1}S_1[(M^\pm(\lambda)+S_1)^{-1}-T_1].
\end{multline*}
Since  by Lemma~\ref{lem:M+S1inv} the operator $(M^\pm(\lambda)+S_1)^{-1}$ satisfies better Lipschitz bounds than $ E_3^\pm(\lambda) $, and since the last two terms are similar, we concentrate on the term $$[(M^\pm(\lambda)+S_1)^{-1}-T_1]S_1B_\pm^{-1}S_1
(M^\pm(\lambda)+S_1)^{-1}.$$
By Lemma~\ref{lem:M+S1inv}, specifically\eqref{eq:M+S1_E2exp}, we have $(M^\pm(\lambda)+S_1)^{-1}=O(1)$, Combining this with \eqref{g def} we see that 
$ (M^\pm(\lambda)+S_1)^{-1}-T_1=O(|\lambda|^{1-})$. Also noting that $B_\pm^{-1}=O(|\lambda|^{-1})$ by Lemma~\ref{lem:Binv},   
we have
$$
[(M^\pm(\lambda)+S_1)^{-1}-T_1]S_1B_\pm^{-1}S_1
(M^\pm(\lambda)+S_1)^{-1} =O(|\lambda|^{0-})=O(|\lambda|^{-1+k}),\,\,0<k<1.
$$
The Lipschitz bound follows by using the bounds above and in addition the bounds in Lemma~\ref{lem:M+S1inv} for $ (M^\pm(\lambda)+S_1)^{-1}$, and by noting that 
$$ 
B_\pm(\lambda_1)^{-1}-B_\pm(\lambda_2)^{-1}= \frac1{\lambda_1}  A^\pm (\lambda_1)^{-1} -\frac1{\lambda_2}  A^\pm (\lambda_2)^{-1}  + E_3^\pm(\lambda_1)  -E_3^\pm(\lambda_2).  
$$
The contribution of $E_3^\pm$   is controlled by the bound in Lemma~\ref{lem:Binv}, specifically \eqref{eq:E3lip12alpha}.  For the contribution of the remaining  terms, we note
$$
	\frac1{\lambda_1}  A^\pm (\lambda_1)^{-1} -\frac1{\lambda_2}  A^\pm (\lambda_2)^{-1}=\bigg( \frac{1}{\lambda_1}-\frac{1}{\lambda_2} \bigg)A^\pm (\lambda_1)^{-1}-\frac{1}{\lambda_2}\bigg(A^\pm (\lambda_2)^{-1}-A^\pm (\lambda_1)^{-1}\bigg).
$$
Then \eqref{eq:Ainv Lip} suffices to control the second term, while the first term is controlled by using $(A^\pm(\lambda))^{-1}=O(1)$ by Proposition~\ref{lem:Ainverse} and the simple bound
$$
	|\lambda_1^{-1}-\lambda_2^{-1}|\les |\lambda_1-\lambda_2|^{\frac12+\gamma}| \lambda_1|^{-\frac32-\gamma}.
$$
\end{proof}

\section{Small energy dispersive estimates when zero is not regular}\label{sec:disp notfree}
In this section we study  the small energy portion of the Stone's formula, \eqref{Stone}, when zero is not regular:
$$
 \int_{-\infty}^{\infty}  e^{-it\lambda} \chi(\lambda)  [\mR_V^+ -\mR_V^-](\lambda)(x,y)\, d\lambda.
$$
In particular, we prove the following result.  
\begin{prop}\label{prop:disp est nonreg}
	Fix $0\leq \gamma<\frac12$. Assume that $|V(x)|\les \la x\ra^{-\beta-}$. 
	If zero is not regular and $\beta>3+2\gamma$, there is a finite-rank operator
		$F_t$ with 
		\begin{align}\label{eqn:pwave low int}
		\sup_{x,y} \left| \int_{-\infty}^{\infty}  e^{-it\lambda} \chi(\lambda)  [\mR_V^+ -\mR_V^-](\lambda)(x,y)\, d\lambda -F_t(x,y)  \right| &\les \la t \ra^{-\f12-\gamma}\la x\ra^\gamma\la y\ra^\gamma,
		\end{align}
		where  $\sup_{t,x,y} |F_t(x,y)|\les 1$, and if $|t|>2$, $\sup_{x,y} |F_t(x,y)|\les( \log |t| )^{-1}$.
		Furthermore, if there is an eigenvalue only at zero, the bound
		\eqref{eqn:pwave low int} remains valid with $F_t=0$.  
\end{prop}

In fact, when zero is not regular we explicitly construct the finite rank operator $F_t$, see \eqref{eqn:Ft def} below.

\begin{proof}[Proof of Proposition~\ref{prop:disp est nonreg}]

Recall \eqref{Rexpnew}. As in the regular case, Lemma~\ref{lem:bornlow2} suffices to control the first few terms arising in \eqref{Rexpnew}, hence we turn our attention to the tail.
Recall that by Lemma~\ref{lem:Minv p} we have  
$$
M^\pm(\lambda)^{-1}= \frac1\lambda S_1 A^\pm (\lambda)^{-1} S_1  + E_4^\pm(\lambda).
$$
The contribution of the second term in the Stone's formula is taken care of by Lemma~\ref{prop:gen} by taking $k=\frac12+\gamma+ $ in the error bounds for $E_4^\pm(\lambda)$. This requires that $\beta>3+2\gamma$.

It remains to consider the  contribution of 
$$
	\frac1\lambda   \mR_0^\pm V\mR_0^\pm v^*S_1A^{\pm}(\lambda)^{-1} S_1 v \mR^{\pm}_0 V\mR_0^\pm.
$$ 
If we replace at least one of the free resolvents with   $\mR_0^\pm -\mG_{0,0}$, we obtain further $\lambda$ smallness which allows us to obtain the desired $\la t\ra^{-\f12}$ bound  with minor modifications of  the  proof of Lemma~\ref{prop:gen}.   In particular, we note that
\begin{multline*}
	\frac1\lambda   \mR_0^\pm V\mR_0^\pm v^*S_1A^{\pm}(\lambda)^{-1} S_1 v \mR^{\pm}_0 V\mR_0^\pm\\
	=\frac1\lambda   \mR_0^\pm V\mG_{0,0} v^*S_1A^{\pm}(\lambda)^{-1} S_1 v \mR^{\pm}_0 V\mR_0^\pm
	+\frac1\lambda   \mR_0^\pm V(\mR_0^\pm-\mG_{0,0}) v^*S_1A^{\pm}(\lambda)^{-1} S_1 v \mR^{\pm}_0 V\mR_0^\pm.
\end{multline*}
Further, 
\begin{multline*}
	\frac1\lambda   \mR_0^\pm V\mG_{0,0} v^*S_1A^{\pm}(\lambda)^{-1} S_1 v \mR^{\pm}_0 V\mR_0^\pm\\
	=\frac1\lambda   \mR_0^\pm V\mG_{0,0} v^*S_1A^{\pm}(\lambda)^{-1} S_1 v \mG_{0,0} V\mR_0^\pm
	+\frac1\lambda   \mR_0^\pm V\mG_{0,0} v^*S_1A^{\pm}(\lambda)^{-1} S_1 v (\mR^{\pm}_0-\mG_{0,0}) V\mR_0^\pm.
\end{multline*}
Iterating this process, we may write  
\begin{multline}\label{mostsing}
	\frac1\lambda   \mR_0^\pm V\mR_0^\pm v^*S_1A^{\pm}(\lambda)^{-1} S_1 v \mR^{\pm}_0 V\mR_0^\pm\\
	=\frac1\lambda  \mG_{0,0} V \mG_{0,0} v^*S_1A^{\pm}(\lambda)^{-1} S_1 v \mG_{0,0} V\mG_{0,0} +\mathcal E_{x,y}(\lambda).  
\end{multline}
We first consider the  contribution of the first term to the Stone's formula.
When there is a p-wave resonance at zero, when $S_1-S_2\neq 0$, using Proposition~\ref{lem:Ainverse}, the $\pm$ difference easily yields a finite rank  term with logarithmic decay in time since
$$
\int_\R e^{-it\lambda} \chi(\lambda) \widetilde O_1\big(\frac1{\lambda\log^2\lambda}\big) d\lambda 
$$	
satisfies the desired bound by Lemma~\ref{lem:log decay}.

So when there is a `p-wave' resonance at zero, we can explicitly construct the operator $F_t$ by
\be	\label{eqn:Ft def}
	F_t:=  \int_{-\infty}^{\infty}  e^{-it\lambda} \chi(\lambda)  \mG_{0,0} V \mG_{0,0} v^*S_1 \bigg(\frac{A^{+}(\lambda)^{-1}-A^{-}(\lambda)^{-1}}{\lambda} \bigg) S_1 v \mG_{0,0} V\mG_{0,0}\, d\lambda.  
\ee

In the eigenvalue only case, when $S_1=S_2\neq0$, by Proposition~\ref{lem:Ainverse}
the leading term in \eqref{mostsing} disappears by $\pm$ cancellation since $A^\pm(\lambda)^{-1}$ is independent of 
the choice of sign in this case. Therefore $F_t=0$.

For the terms in $\mathcal E_{x,y}(\lambda)$, we have the following variant of Lemma~\ref{prop:gen} (we drop the $\pm$ signs since we don't rely on cancellation):
\begin{lemma}\label{prop:gen2} Fix $0\leq \gamma<\frac12$. Assume that $|V(x)|\les \la x\ra^{-2-2\gamma-}$. Let $T(\lambda)$ be an absolutely bounded operator  satisfying
(for $|\lambda|,|\lambda_1|,|\lambda_2|\les 1$ with $|\lambda_1|\leq |\lambda_2|$) 
$$
\big\| |T(\lambda ) |\big\|_{L^2\to L^2} \les |\lambda|^{-1},
$$
$$
\big\| |T(\lambda_1)-T(\lambda_2)|\big\|_{L^2\to L^2}\les |\lambda_1|^{- \frac32-\gamma} |\lambda_1-\lambda_2|^{\frac12+\gamma}.
$$
 Then
$$
\Big|\int_\R e^{-it\lambda} \chi(\lambda)\big[\mR_1    V \mR_2     v^*T   v \mR_3   V  \mR_4 \big](\lambda)(x,y) d\lambda \Big|\les \la t\ra^{-\frac12-\gamma} \la x\ra^{\gamma} \la y \ra^\gamma,
$$
where  $\mR_j=\mR_0$, $\mG_{0,0}$, or $\mR_0-\mG_{0,0}$, $j=1,2,3,4,$ and at least one of them is  $\mR_0-\mG_{0,0}$.
\end{lemma}
Note that the hypothesis  is satisfied by the mean value theorem if  $T(\lambda)=\widetilde O_1(\lambda^{-1})$ as  an absolutely bounded operator, in particular when $T(\lambda)=\tfrac1\lambda A^\pm(\lambda)^{-1}$.  
\begin{proof}
Let $\widetilde R:= v \mR_3  V  \mR_4$.
Since each $\mR_0$, $\mG_{0,0}$, and $\mR_0-\mG_{0,0}$ satisfies the bounds \eqref{eq:R0simple} and \eqref{R0lip12alpha}, the operator $\widetilde R$ satisfies the  bounds \eqref{tildeR} and \eqref{tildeRlip} in the proof of Lemma~\ref{prop:gen}.  In particular, the  $L^2_{y_1}$ norm of $\widetilde R(\lambda)(y_1,y)$ is bounded in  $y$ and $\lambda$, and the $L^2_{y_1}$ norm of $\widetilde R(\lambda_1)(y_1,y)-\widetilde R(\lambda_2)(y_1,y)$ is bounded by  $\la y\ra^{\gamma} |\lambda_1-\lambda_2|^{\frac12+\gamma} |\lambda_2|^{\frac12-\gamma-}$.

If $\mR_3$ or $\mR_4$ is equal to $\mR_0-\mG_{0,0}$. Then,   by \eqref{eq:R0simple}, 
$$ \mR_0-\mG_{0,0}= O(|\lambda|^{1-}(1+|x-y|^{0-}). $$ 
Therefore $\widetilde R$ satisfies the following improved pointwise bound 
\be\label{tildeRimp}
|\widetilde R(\lambda)(y_1,y)| \les |\lambda|^{1-}  (1+|y_1-y|^{0-})\la y_1\ra^{-1-}.
\ee
In particular, the  $L^2_{y_1}$ norm of $\widetilde R(\lambda)(y_1,y)$ is bounded by 
$|\lambda|^{1-}$.

Using these bounds and the hypothesis for $T$, we see that (with $\Gamma:=\mR_1  V \mR_2  v^*T   v \mR_3   V  \mR_4$)
$$
|\Gamma|\les |\lambda|^{0-}.
$$
This implies the uniform bound when $t$ is small. Also
using this  in the case $|\lambda_1-\lambda_2|\gtrsim |\lambda_2|$ we obtain 
$$
|\Gamma(\lambda_1)-\Gamma(\lambda_2)|\les |\lambda_1-\lambda_2|^{\frac12+\gamma} |\lambda_1|^{-\frac12-\gamma-}
$$
When  $|\lambda_1-\lambda_2|\ll|\lambda_2|\approx |\lambda_1|,$ we estimate  
$\Gamma(\lambda_1)-\Gamma(\lambda_2)$ by 
\begin{multline*}
|\lambda_1|^{1-} |\lambda_1|^{-\frac32-\gamma} |\lambda_1-\lambda_2|^{\frac12+\gamma} 
+|\lambda_1-\lambda_2|^{\frac12+\gamma} |\lambda_1|^{\frac12-\gamma-} \la x\ra^\gamma   \la y\ra^\gamma |\lambda_1|^{-1} \\
\les |\lambda_1-\lambda_2|^{\frac12+\gamma} |\lambda_1|^{-\frac12-\gamma-} \la x\ra^\gamma   \la y\ra^\gamma.  
\end{multline*}
The first summand above corresponds to the case when the difference is on $T$ and the second summand corresponds to the remaining cases. Combining these bounds for $0\leq \gamma<\frac12$  we have 
$$
|\Gamma(\lambda_1)-\Gamma(\lambda_2)|\les \la x\ra^{\gamma}\la y\ra^{\gamma} |\lambda_1-\lambda_2|^{\frac12+\gamma} |\lambda_1|^{-1+},\,\,|\lambda_j|\ll 1, j=1,2.
$$
Therefore, by applying the Lipschitz argument as in \eqref{lip12trick}, we bound the integral  by
$  \la t\ra^{-\frac12-\gamma} \la x\ra^{\gamma}\la y\ra^{\gamma}. $   
\end{proof}

This finishes the proof of Proposition~\ref{prop:disp est nonreg}.
\end{proof}

\section{Threshold characterization}\label{sec:TC}

The characterization of the threshold is similar to the characterization for the massive case in \cite{egd}.  See \cite{EGT} for the three dimensional threshold characterization.  These results have roots in the characterizations for Schr\"odinger operators may be found in \cite{ES,eg2,EGG4}.

\begin{lemma} \label{lem:S1char} Assume that $|V(x)|\les \la x\ra^{-\beta}$ for some $\beta>2$.	
	If $\phi \in $ ker$(T)$, then $\phi =Uv\psi $ with $\psi$ a distributional solution to $H\psi=0$ and $\psi\in L^p(\R^2)$ for all $p>2$.
	
\end{lemma}

\begin{proof}
	
	Take $\phi \in $ ker$(T)$, $\phi \in L^2$.  Then
	$$
	0=T\phi =U\phi +v\mG_{0,0}v^* \phi =0 \quad \Rightarrow \quad \phi =-Uv\mG_{0,0}v^* \phi .
	$$
	Define $\psi:=-\mG_{0,0}v^*\phi $, then $\phi =Uv\psi$.  Now, with $H=D_0+V=-i\alpha \cdot \nabla+V$,
	$$
	H\psi=(-i\alpha \cdot \nabla+V)\psi=-i\alpha \cdot \nabla \psi+v^*Uv\psi
	=i\alpha \cdot \nabla(\mG_{0,0}v^*\phi )+v^*\phi 
	$$
	Here, recalling \eqref{G00def} and \eqref{eqn:anticomm}, we have
	$$
	i\alpha \cdot \nabla(\mG_{0,0}v^*\phi )=i\alpha \cdot \nabla(-i\alpha \cdot \nabla G_0 v^*\phi )=\Delta (-\Delta)^{-1}v^*\phi =-v^*\phi 
	$$
	distributionally.  So,
	$$
	H\psi=  i\alpha \cdot \nabla(\mG_{0,0}v^*\phi )+v^*\phi =-v^*\phi +v^*\phi =0.
	$$
	That is, if $\phi \in$ ker$(T)$ we have $H\psi=0$.  Now, to show that $\psi\in L^p$, we have $\psi=-\mG_{0,0}v^*\phi $ with $\phi \in L^2$.  We can bound $|\mG_{0,0}(x,y)|\les |x-y|^{-1}$ to employ a fractional integral operator argument.  So that,
	$$
	\|\psi\|_q=
	\|\mG_{0,0}v^*\phi  \|_q \les \bigg\| \int_{\R^2} \frac{\la y\ra^{-1-}}{|x-y|} |\phi (y)|\, dy		\bigg\|_q  \les \|\phi  \|_2
	$$
	for $2<q<\infty$.  Furthermore, since $\phi =Uv\psi$ we have 
	$\psi=-\mG_{0,0}V\psi$, and
	$$
	|\psi|\leq |\mG_{0,0}V\psi|\les \int_{\R^2} \frac{\la y\ra^{-2-}}{|x-y|} |\psi(y)|\, dy \les \|\psi\|_{3} \| |x-\cdot|^{-1} \la \cdot \ra^{-2-}\|_{\f32}\les 1.
	$$
	Thus, $\psi\in L^p$ for all $p>2$.

\end{proof}

\begin{lemma}\label{lem:S1char2} Assume that $|V(x)|\les \la x\ra^{-\beta}$ for some $\beta>2$.	
	If $\phi=Uv\psi\in S_1L^2$ then
	$$
	\psi(x)=\frac{-i\alpha \cdot x}{2\pi \la x\ra^2}\mG_{1,1} v^* \phi  +\Gamma_2
	$$
	where $\Gamma_2\in L^2\cap L^\infty$.
\end{lemma}

\begin{proof}
	
	By the last lemma, we have $\psi \in L^\infty$.  We recall that $\psi=-\mG_{0,0}v^* \phi$ and the kernel of $\mG_{1,1}$ is $1$, so
	\begin{multline*}
	\psi(x)=-\frac{i}{2\pi} \int_{\R^2} \frac{\alpha \cdot (x-y)}{|x-y|^2}v^*(y) \phi(y)\, dy\\
	=-\frac{i}{2\pi} \int_{\R^2} \bigg[\frac{\alpha \cdot (x-y)}{|x-y|^2}-\frac{\alpha \cdot x}{\la x\ra^2} \bigg]v^*(y) \phi(y)\, dy
	-\frac{i\alpha \cdot x}{2\pi \la x\ra^2} \mG_{1,1}v^* \phi.
	\end{multline*}
	The first term is in $L^2$ (see Lemma~7.3 in \cite{egd}).  Combining this with $\psi\in L^\infty$ finishes the proof.   We note that the assumption that $\beta>2$ suffices here, the logarithmic terms in the massive case considered in \cite{egd} required further decay of the potential.  These terms do not occur in the massless case, specifically we need only (68) in \cite{egd} for which $\beta>2$ is sufficient.

\end{proof}

\begin{corollary}\label{cor:dim}
	
	The rank of $S_1$ is at most two plus the dimension of the eigenspace at zero.
	
\end{corollary}

We note that the at most two dimensional space of resonances correspond to the p-wave resonances in the massive Dirac, \cite{egd}, and Schr\"odinger \cite{eg2} operators.  We again note that there are no `s-wave' resonances in the massless case.

\begin{lemma}
	Assume that $|V(x)|\les \la x\ra^{-\beta}$ for some $\beta>2$.	
	If $H\psi=0$ with $\psi\in L^2+\cap_{p\in (2,\infty]} L^p$, then $\phi=Uv\psi\in S_1L^2$, i.e. $T\phi=0$.
	
\end{lemma}

\begin{proof}
	
	Using $H\psi=0$, we have $i\alpha \cdot \nabla \psi =V\psi =v^* \phi$.  We first show that $\psi=-\mG_{0,0}v^*\phi$.  Since $\phi=Uv\psi \in L^2$, we have that $v^*\phi \in L^1$.  Recalling \eqref{G00def}, $\mG_{0,0}=-i\alpha \cdot \nabla G_0$, so
	\begin{align*}
	-i\alpha \cdot \nabla \big[ \psi +\mG_{0,0}v^* \phi \big]
	&= -i\alpha \cdot \nabla \psi+ \Delta G_0 v^* \phi=v^*\phi-v^*\phi =0.
	\end{align*}
	Thus,
	$$
	-i\alpha \cdot \nabla \big[ \psi +\mG_{0,0}v^* \phi \big]=0 \quad
	\Rightarrow \quad  \psi +\mG_{0,0}v^* \phi = (c_1, c_2)^T.
	$$
	Since $\psi \in L^2+\cap_{p\in (2,\infty]} L^p$ and $\mG_{0,0}v^*\phi\in L^p$ for all $p>2$ by the proof of Lemma~\ref{lem:S1char}, we have  $(c_1,c_2)^T\in L^2+ \cap_{p\in (2,\infty]}L^p$. Therefore,    $c_1=c_2=0$, and $\psi=-\mG_{0,0}v^*\phi$ as desired.  
	
	Next, to show $T\phi=0$, we note that $U\phi=U^2v\psi =v\psi$. Also 
	recalling that $T=U+v\mG_{0,0}v^*$ and $\psi=-\mG_{0,0}v^*\phi$, we have 
	$$
	T\phi = U \phi +v\mG_{0,0}v^* \phi= v\psi-v\psi=0.
	$$
\end{proof}

Recall that $S_2$ is the projection onto the kernel of $S_1v\mG_{1,1}v^* S_1$. We have the following classification for $S_2L^2$:

\begin{lemma}\label{lem:S2class} Assume that $|V(x)|\les \la x\ra^{-\beta}$ for some $\beta>2$.	 Fix $\phi=Uv\psi \in S_1L^2$.  Then $\phi\in S_2 L^2$ if and only if 
	$\psi \in L^2$.
	
\end{lemma}

\begin{proof}
	
	By Lemma~\ref{lem:S1char2}, $\psi \in L^2$ if and only if $\mG_{1,1}v^*\phi=0$, which is equivalent to $\phi$ being in the kernel of $S_1v \mG_{1,1}v^*S_1$.

\end{proof}

We now prove that  $S_2 v \mG_{1,0}v^* S_2$ is always invertible on $S_2L^2$.

\begin{lemma}\label{lem:S2invert} Assume that $|V(x)|\les \la x\ra^{-\beta}$ for some $\beta>2$.	For $\phi\in S_2L^2$, we have the identity
\be\label{eq:G00G10}
	\la \mG_{0,0}v^*\phi, \mG_{0,0}v^*\phi \ra = \la v^*\phi, \mG_{1,0}v^*\phi \ra.
	\ee
	Furthermore, the kernel of $S_2v\mG_{1,0}v^* S_2$ is trivial.
	
\end{lemma}

\begin{proof}
	First note that by Lemma~\ref{lem:S2class}, we have $\psi=-\mG_{0,0} v^* \phi\in L^2$.
	On the Fourier side,	
	\begin{multline*}
	\la \mG_{0,0}v^* \phi, \mG_{0,0} v^* \phi\ra 
	 = \int_{\R^2} \frac{1}{|\xi|^4}\bigg\la \left(\begin{array}{cc}
	0 & \overline \xi \\ \xi & 0
	\end{array} \right) \widehat{ v^* \phi}, \left(\begin{array}{cc}
	0 & \overline \xi \\ \xi & 0
	\end{array} \right) \widehat{ v^* \phi}
	\bigg \ra_{\C^2} \, d\xi \\ =\int_{\R^2} \frac{1}{|\xi|^2}\la \widehat{ v^* \phi},  \widehat{ v^* \phi}
	\ra_{\C^2} \, d\xi.
	\end{multline*}
 Using the expansion  $R_0(-\lambda^2)=g(\lambda) \mG_{1,1}+G_0+O(\lambda^{0+}|x-y|^{0+})$  and   $\mG_{1,1}v^*\phi=0$, we have
	\begin{multline*}
	 \la v^*\phi, \mG_{1,0}v^*\phi \ra= \la v^*\phi,  G_{0}v^*\phi \ra= \lim_{\lambda\to 0}
	 \la v^*\phi,  R_0(-\lambda^2) v^*\phi \ra \\
	 = \lim_{\lambda\to 0}\int_{\R^2} \frac{1}{|\xi|^2+\lambda^2}\la \widehat{ v^* \phi},  \widehat{ v^* \phi}
	\ra_{\C^2} \, d\xi =\int_{\R^2} \frac{1}{|\xi|^2}\la \widehat{ v^* \phi},  \widehat{ v^* \phi}
	\ra_{\C^2} \, d\xi
	\end{multline*}
	by monotone convergence theorem. This implies the identity \eqref{eq:G00G10}.

	Take $\phi$ in the kernel of $S_2v\mG_{1,0}v^* S_2$, then by the identity \eqref{eq:G00G10},
	$$\|\psi\|_{L^2}^2=\la \mG_{0,0}v^* \phi, \mG_{0,0} v^* \phi\ra =0.$$ Thus, $\psi=0$ and $\phi=Uv\psi  =0$.  
	\end{proof}

\begin{lemma}\label{lem:eigenproj}
	
	The projection onto the zero energy eigenspace is
	$$
	P_0=\mG_{0,0}v S_2 [S_2 v\mG_{1,0} v^* S_2]^{-1} S_2 v^* \mG_{0,0}.
	$$
	
\end{lemma}

The proof follows along the lines of Lemma~7.10 in \cite{egd}.  For the sake of brevity, we omit the proof.

\section{High Energy Dispersive estimates}\label{sec:high energy}

We now provide a proof of Theorem~\ref{thm:main_hi}, the high energy dispersive estimate.  The theorem follows from

\begin{prop}\label{prop:hi}
	
	Under the hypotheses of Theorem~\ref{thm:main_hi}, we have the bound
	\be\label{prophi}
		\sup_{x,y} \bigg| \int_{\R}  e^{-it\lambda} \lambda^{-2-} \widetilde \chi(\lambda) [\mR_V^+-\mR_V^- ](\lambda)(x,y)\, d\lambda \bigg| \les \la t \ra^{-\f12},
	\ee
	provided $|V(x)|\les \la x\ra^{-2-}$.
	Furthermore, for $0\leq \gamma\leq \frac32$  we have
		$$
			\bigg| \int_{\R} e^{-it\lambda} \lambda^{-2-} \widetilde \chi(\lambda) [\mR_V^+-\mR_V^- ](\lambda)(x,y)\, d\lambda \bigg| \les \la x \ra^{\gamma} \la y \ra ^{\gamma} \la t \ra^{-\f12-\gamma},
		$$ 
		provided $|V(x)|\les \la x\ra^{-\beta}$ for some $\beta>\min(2+2\gamma,3). $
\end{prop}

We employ the resolvent identity twice to write
\be\label{eq:RV high}
	\mR_V^\pm(\lambda)=\mR_0^\pm(\lambda)-\mR_0^\pm(\lambda)V\mR_0^\pm(\lambda)+\mR_0^\pm(\lambda)V\mR_V^\pm(\lambda)V\mR_0^\pm(\lambda).
\ee
By virtue of Theorem~\ref{thm:free}, we need only bound the second and third summands.

\begin{lemma}\label{lem:hi term2}
	
	The contribution of the second term in \eqref{eq:RV high} to \eqref{prophi} satisfies the decay bounds in Proposition~\ref{prop:hi}.   
	
\end{lemma}

\begin{proof}
	
	We write the free resolvents as $\mR_0^\pm(\lambda)(x,y)=\mR_L^\pm(\lambda|x-y|)+\mR_H^\pm(\lambda|x-y|)$, where
	$\mR_L^\pm(\lambda|x-y|)=\chi(\lambda|x-y|)\mR_0^\pm(\lambda|x-y|)$ and $\mR_H^\pm(\lambda|x-y|)=\widetilde\chi(\lambda|x-y|)\mR_0^\pm(\lambda|x-y|)$.  We consider the contributions of terms with at least one instance of $\mR_H^\pm$, such as
	\be \label{eq:hi mix}
		\int_{ \R}\int_{ \R^2} e^{-it\lambda} \lambda^{-2-} \widetilde \chi(\lambda) [\mR_L^\pm (\lambda|x-x_1|) +\mR_H^\pm(\lambda|x-x_1|)] V(x_1)\mR_H^\pm(\lambda|x_1-y|)\, dx_1 d\lambda,
	\ee
	and the two terms with only the low-energy part of the resolvents,
	$$
		\int_{ \R}\int_{ \R^2} e^{-it\lambda} \lambda^{-2-} \widetilde \chi(\lambda) [\mR_L^+ (\lambda|x-x_1|)\mR_L^+(\lambda|x_1-y|) -\mR_L ^- (\lambda|x-x_1|)\mR_L^-(\lambda|x_1-y|)]V(x_1) \,dx_1 d\lambda.
	$$	
	That is, we need only use the `+/-' on the `low-low' term.  We consider the `low-low' term first.  By symmetry, we need only consider
	\be\label{eq:hi lowlow}
		\int_{ \R}\int_{ \R^2} e^{-it\lambda} \lambda^{-2-} \widetilde \chi(\lambda) [\mR_L^+ (\lambda|x-x_1|) -\mR_L ^- (\lambda|x-x_1|)]V(x_1) \mR_L^-(\lambda|x_1-y|) \, dx_1 d\lambda.
	\ee
	From \eqref{R0pm} and the support condition $\lambda|x-x_1|\ll 1$, we see that 
	$$
		|\partial_\lambda^k [\mR_L^+ (\lambda|x-x_1|) -\mR_L ^- (\lambda|x-x_1|)]|\les \lambda^{1-k}, \qquad k=0,1,2.
	$$
	While from Lemma~\ref{lem:R0simpleee} and the support condition, we see that
	$$
		|\mR_L^\pm|\les \frac{1}{|x_1-y|}, \qquad |\partial_\lambda \mR_L^\pm|\les \frac{1}{(\lambda |x_1-y|)^{0+}}, \qquad |\partial_\lambda^2 \mR_L^\pm|\les \frac{1}{ \lambda (\lambda|x_1-y|)^{0+}},
	$$
	which implies that
	\begin{align*}
		|\eqref{eq:hi lowlow}| \les \int_{ \R } \int_{ \R^2}\lambda^{-1-}\widetilde \chi(\lambda) \frac{|V(x_1)|}{|x_1-y|}\,dx_1 d\lambda.
	\end{align*}
	We can see that the integral is bounded uniformly in $x$ and $y$ by Lemma~\ref{lem:EGcor}.  For $|t|>1$, by a single integration by parts, one has
	\begin{align*}
		|\eqref{eq:hi lowlow}| \les  \frac{1}{|t|} \int_{ \R }\int_{ \R^2} \lambda^{-2-}\widetilde \chi(\lambda) |V(x_1)| \bigg( \frac{1}{|x_1-y|}+\frac{\lambda}{(\lambda |x-x_1|)^{0+}} \bigg)   \, dx_1 d\lambda \les \frac{1}{|t|}.
	\end{align*}	
	There are no boundary terms due to the cut-off.
	While integrating by parts twice yields
	\begin{align*}
		|\eqref{eq:hi lowlow}| \les  \frac{1}{t^2} \int_{ \R } \lambda^{-2-}\widetilde \chi(\lambda) |V(x_1)|   \frac{1}{ \lambda|x_1-y| }  \,dx_1 d\lambda \les \frac{1}{t^2}.
	\end{align*}	
	We now turn to the contribution of \eqref{eq:hi mix}, which necessitates   spatial weights for faster  time decay.  We have to control two terms. We first look at  the `low-high' interaction:
	\be \label{eq:hi hilow}
		\int_{ \R } \int_{\R^2}  e^{-it\lambda\pm i\lambda |x_1-y|} \lambda^{-2-} \widetilde \chi(\lambda) \mR_L^\pm(\lambda|x-x_1|) V(x_1) \widetilde \omega_\pm (\lambda |x_1-y|)\,dx_1 d\lambda.
	\ee
	Using \eqref{eq:omegatilde}, we  see that
	$$
		|\eqref{eq:hi hilow}| \les \int_{\R } \int_{\R^2}\lambda^{-\f32-} \widetilde \chi(\lambda) \frac{|V(x_1)|}{|x-x_1| |x_1-y|^{\f12}}\, d x_1d\lambda \les 1.
	$$
	The spatial integral is bounded by Lemma~\ref{lem:EGcor} with $k=1$ and $\ell=\f12$.
	
For $t>1$,	without loss of generality we work with the `+' case. We consider two subcases based on the size of   $t- |x_1-y|$.     In the case that $|t-|x_1-y||\leq \frac{t}{2}$, we have that $|x_1-y|\gtrsim t$. Using \eqref{eq:omegatilde}  we have $|\widetilde \omega_\pm(\lambda |x_1-y|)|\les \lambda^{\f12} t^{-\f12}$.  Then,
	$$
		|\eqref{eq:hi hilow}| \les  t^{-\f12} \int_{ \R }\int_{\R^2} \lambda^{-\f32-}\widetilde \chi(\lambda) \frac{|V(x_1)|}{|x-x_1|} \, dx_1 d\lambda \les t^{-\f12}. 
	$$
	Here the spatial integrals are controlled by Lemma~\ref{lem:EGcor}.  At the cost of spatial weights, one may attain faster time decay.  Furthermore, since $|x_1-y|\gtrsim t$, for $\gamma>0$ we have
	$$
		1\les \frac{|x_1-y|^\gamma}{t^\gamma}\les \frac{\la x_1\ra^\gamma \la y \ra^\gamma }{t^\gamma}.
	$$
	Thus,   
	\begin{multline*}
		|\eqref{eq:hi hilow}| \les  t^{-\f12} \int_{ \R } \int_{\R^2} \lambda^{-\f32-}\widetilde \chi(\lambda) \frac{|V(x_1)|}{|x-x_1|} \, dx_1 d\lambda \\
		\les  t^{-\f12-\gamma} \la y \ra^\gamma \int_{ \R } \int_{\R^2}\lambda^{-\f32-}\widetilde \chi(\lambda) \frac{|V(x_1)|\la x_1 \ra^\gamma}{|x-x_1|} \, dx_1 d\lambda
		\les  t^{-\f12-\gamma} \la y \ra^\gamma.
	\end{multline*}
	Provided $V$ decays sufficiently, Lemma~\ref{lem:EGcor} controls the spatial integrals.
	
	On the other hand, if $|t-|x_1-y||\geq \frac{t}{2}$ we integrate by parts twice.  There are no boundary terms due to the support of the cut-off, and we see
	\begin{multline*}
		 |\eqref{eq:hi hilow}|\les  \frac{1}{(t-|x_1-y|)^2}\int_{ \R }\int_{\R^2} \big| \partial_\lambda^2 \big[\lambda^{-2-} \widetilde \chi(\lambda) \mR_L^\pm(\lambda|x-x_1|) V(x_1) \widetilde \omega_\pm (\lambda |x_1-y|)\big]\big|\,  dx_1 d\lambda \\
		\les \frac{1}{t^2} \int_{ \R }\int_{\R^2} \lambda^{-\f32-} \widetilde \chi(\lambda) \frac{|V(x_1)|}{|x_1-x||y-x_1|^\f12} \, dx_1d\lambda \les \frac{1}{t^2}.
	\end{multline*}
	The final term to consider is the `high-high' interaction in \eqref{eq:hi mix}.  
	\be\label{eq:hi hihi}
		\int_{\R} \int_{\R^2} e^{-it\lambda \pm i\lambda (|x-x_1|+|x_1-y|)} \lambda^{-2-} \widetilde \chi(\lambda) \widetilde \omega_\pm (\lambda |x-x_1|) V(x_1)\omega_\pm (\lambda |x_1-y|)\, dx_1d\lambda
	\ee
	We consider the `+' case.  The integral is bounded in $t$ as before. For $t>1$, first we consider when $|t- |x-x_1|-|x_1-y| |\leq \frac{t}{2}$. In this case  we have that $\max \{ |x-x_1|,|x_1-y| \} \gtrsim t$.  The analysis then proceeds as in the bounds for \eqref{eq:hi hilow} in the analogous case.    
	
	Finally, if $|t- |x-x_1|-|x_1-y||\geq \frac{t}{2}$ we may integrate by parts twice to obtain   
	\begin{multline*}
		|\eqref{eq:hi hihi}| \les \frac{1}{(t- |x-x_1|-|x_1-y| )^2} \int_{\R} \int_{\R^2} \big| \partial_\lambda^2 [\lambda^{-2-}\widetilde \omega_\pm (\lambda |x-x_1|) V(x_1)\omega_\pm (\lambda |x_1-y|)]\big|\, d x_1 d\lambda \\
		\les \frac{1}{t^2}\int_{\R}\int_{\R^2} \lambda^{-3-} \frac{|V(x_1)|}{ |x-x_1|^{\f12}|x_1-y|^{\f12} }\, dx_1d\lambda \les t^{-2}.
	\end{multline*}
	This finishes the proof.

\end{proof}

\begin{lemma}\label{lem:hi term3}
 The contribution of the third term in \eqref{eq:RV high} to \eqref{prophi} satisfies the decay bounds in Proposition~\ref{prop:hi}.      
\end{lemma}
\begin{proof} We drop the $\pm$ signs since one can not use the $\pm$ cancellation in this case, and we consider only the `+' case. 

Under the hypotheses of Theorem~\ref{thm:main_hi}, we have the limiting absorption principle
	$$
		\sup_{\lambda>0} \|\partial_\lambda^k \mR_V^\pm(\lambda)\|_{L^{2,\sigma+k}\to L^{2,-\sigma-k}} \les 1, \qquad \sigma>\f12, \qquad k=0,1,2,
	$$	
	see \cite{EGG}.
	The main obstacle at the moment is that $\mR_0(\lambda)(x,\cdot)$ is not locally in $L^2$ due to the $\frac{1}{|x - \cdot|}$ singularity in $\mR_L$.  We write
$$
		\mR_0 V\mR_V V\mR_0 = \mR_L V\mR_V V \mR_L+ \mR_H V \mR_V V \mR_L + \mR_L V \mR_V V \mR_H+ \mR_H V \mR_V V \mR_H.
$$	 
	Now, using the resolvent identity on $\mR_V$ for terms involving $\mR_L$ we have
	\begin{multline}\label{eq:RV hi iterated}
		\mR_0 V\mR_V  V\mR_0 \\
		=\mR_L V[ \mR_0 -\mR_0 V \mR_0+\mR_0 V \mR_V V \mR_0] V \mR_L+ \mR_H V[\mR_0- \mR_V V \mR_0  ]V\mR_L\\
		 + \mR_L V[\mR_0- \mR_0 V \mR_V  ] V \mR_H+ \mR_H V \mR_V V \mR_H.
	\end{multline}	
	The summands that do not contain $\mR_V$ follow roughly the same argument as in the previous lemma. Cancellation between terms (as in~\eqref{eq:hi lowlow}) is not needed even at low energy because the iterated integral $\iint_{\R^4} \frac{V(x_1)V(x_2)}{|x-x_1| |x_1 - x_2| |x_2-y|}\,dx_1 dx_2$ is bounded uniformly in $x$ and $y$.  

	We consider first the contribution of the final term in \eqref{eq:RV hi iterated}:
	\begin{multline}\label{eq:RV hihi}
		\int_{ \R } \frac{e^{-i t\lambda}  \widetilde \chi(\lambda)}{ \lambda^{2+}} \mR_H V \mR_V V \mR_H(\lambda)(x,y)\, d\lambda\\
		=\int_{ \R^5 } \frac{e^{-i \lambda (t- |y-x_1|-|x_2-x|)}  \widetilde \chi(\lambda)}{\lambda^{2+}} \omega_\pm (\lambda |x-x_1|) [V   \mR_V V](x_1,x_2) \omega_\pm (\lambda |x_2-y|)\, dx_1 dx_2d\lambda.
	\end{multline}
	The boundedness of this integral follows from the bound in \eqref{eq:omegatilde}: 
	
	\begin{multline*}
		|\eqref{eq:RV hihi}|  
		\les  \int_\R \frac{ \widetilde \chi(\lambda)}{\lambda^{2+}}
		  \| \omega_+ (\lambda |x-\cdot|) V(\cdot)\|_{L^{2,\f12+} } \| \mR_V\|_{L^{2,\f12+}\to L^{2,-\f12-}} \| \omega_+ (\lambda |\cdot-y|) V(\cdot)\|_{L^{2,\f12+} } d\lambda 
		 \\ \les \int_\R \frac{ \widetilde \chi(\lambda)}{\lambda^{1+}} d\lambda \les 1.   
	\end{multline*}
 To show the time decay, we do an analysis as in the proof of the `high-high' term in Lemma~\ref{lem:hi term2}.    Let $\psi_1, \psi_2$ be a partition of unity such that $\psi_1 (z)+\psi_2 (z)=1$ with $\psi_2$ supported on $|z|\gtrsim 1$ and $\psi_1$ on $|z|\ll 1$.  Using \eqref{eq:omegatilde}  we see that
$$
		\| \psi_2( |x-x_1|/t)   \omega_+ (\lambda |x-x_1|) V(x_1)\|_{L^{2,\f12+}_{x_1}} \les \lambda^{\frac12}t^{-\frac12} 
		$$
		$$
		   \|  \omega_+ (\lambda |x_2-y|) V(x_2)\|_{L^{2,\f12+}_{x_2}}\les  \lambda ^{\f12}. 
$$	 
Also using the limiting absorption principle, we estimate   \eqref{eq:RV hihi} in this case by
$$t^{-\frac12}\int_\R \frac{ \widetilde \chi(\lambda)}{\lambda^{1+}} d\lambda \les t^{-\frac12}.
$$
	To obtain the faster decay,  note that for any $\gamma\geq 0$,
	\begin{multline*} 
	\| \psi_2( |x-x_1|/t)   \omega_+ (\lambda |x-x_1|) V(x_1)\|_{L^{2,\f12+}_{x_1}} \\ \les 	\| \psi_2( |x-x_1|/t) \bigg(\frac{|x-x_1|}{t}\bigg)^\gamma  \omega_+ (\lambda |x-x_1|) V(x_1)\|_{L^{2,\f12+}_{x_1}} \les  \lambda^{\frac12} \la x\ra^\gamma t^{-\frac12-\gamma}. 
\end{multline*}
The case $\psi_1(|x-x_1|/t) \psi_2(|x_2-y|/t)$ is treated similarly.
It remains to consider the contribution of $\psi_1(|x-x_1|/t) \psi_1(|x_2-y|/t)$, which implies that   $t- |x-x_1|-|x_2-y| \geq \frac{t}{2}$. Therefore, we can integrate by parts twice to obtain (with $\omega_1(\lambda,t,|x-x_1|) :=\psi_1(|x-x_1|/t) \omega_+ (\lambda |x-x_1|)$)
	\begin{multline*}
		|\eqref{eq:RV hihi}|\les \frac{1}{t^2} \int_{ \R } \bigg| \int_{\R^4} \partial_\lambda^2 \big(  \frac{\widetilde \chi(\lambda)}{\lambda^{2+}} \omega_1(\lambda,t,|x-x_1|)   [V \mR_V V](x_1,x_2)\omega_1(\lambda,t,|x_2-y|)  \big) dx_1dx_2\bigg|\, d\lambda \\
		\les \frac{1 }{t^2} \int_{\R} \lambda^{-1-} \widetilde \chi(\lambda) \sum_{\ell=0}^{2} \bigg\| \frac{V(x_1) \la x_1 \ra^{-\f12+\ell+}}{|x-x_1|^\f12}\bigg\|_{L^2_{x_1}} \| \partial_{\lambda}^\ell \mR_V \|_{L^{2, \f12+\ell+} \to L^{2,-\f12-\ell-}} \\ \times\bigg\| \frac{V(x_2) \la x_2 \ra^{-\f12+\ell+}}{|x_2-y|^\f12}\bigg\|_{L^2_{x_2}}  \, d\lambda \les \frac{1}{t^2}.
	\end{multline*}
	We turn now to a `low-low' interaction term in \eqref{eq:RV hi iterated}:
	\begin{align}\label{eq:RV lowlow}
		\int_{\R} e^{-it\lambda} \lambda^{-2-} \widetilde \chi(\lambda) \mR_L V\mR_0 V \mR_V V \mR_0 V \mR_L (\lambda)\, d\lambda
	\end{align}
	For the inner resolvents we use the following bounds for $\lambda \gtrsim 1$, which are not sharp but suffice for our purposes
	$$
		\big| \partial^k_\lambda \mR_0(\lambda)(x,y) \big|\les \lambda^{\f12} \bigg( \frac{1}{|x-y|}+|x-y|^{k-\f12} \bigg), \qquad k=0,1,2
	$$
	The boundedness and time decay follows from the limiting absorption principle and the observation that
	$$
		\|\partial^k_\lambda \mR_L V \mR_0 V(x, \cdot)\|_{L^{2,\f12+}}\les \lambda^{\f12}.
	$$
The remaining terms in \eqref{eq:RV hi iterated} can be treated similarly. 
\end{proof}

\section{Integral Estimates}\label{sec:int ests}

Finally, we provide proof of the integral estimates that are used throughout the paper.  We first provide the time decay estimate.

\begin{lemma}\label{lem:log decay}
	
	We have the bound
	$$
		\bigg| \int_\R e^{-it\lambda} \chi(\lambda) \widetilde O_1\bigg(\frac1{\lambda\log^2\lambda}\bigg) d\lambda  \bigg| \les 
		\left\{\begin{array}{ll} 1 & \text{for all } t,\\
		(\log |t|)^{-1} & |t|> 2. \end{array} \right.
	$$
	
\end{lemma}

\begin{proof}
	
	The boundedness of the integral follows from the integrability of $(\lambda \log^2 \lambda)^{-1}$ on the support of $\chi$.  The large $|t|$ decay follows by dividing the integral into $|\lambda|<|t|^{-1}$ and integrating by parts when $|\lambda|\geq |t|^{-1}$, see Lemma~3.2 in \cite{EGG4}. 
\end{proof}
 
Now, we catalog the spatial integral estimates we use.  The first bound is a special case of Lemma~6.3 in \cite{EG1}. 
\begin{lemma}\label{lem:EGcor}
	
	For $0\leq k, \ell <2$, $\beta>0$ so that $k+\ell+\beta\geq 2$ and $k+\ell \neq 2$,
	$$
		\int_{ \R^2} \frac{\la x_1\ra^{-\beta -}}{|x-x_1|^k |x_1-y|^\ell  }\, dx_1 \les \bigg( \frac{1}{|x-y|} \bigg)^{\max \{0,k+\ell-2 \}}.
	$$
	
\end{lemma}

We also state the following corollaries:

\begin{lemma}\label{lem:spatial int1alpha}
	
	For any $\gamma \geq 0$, we have
	$$
	\int_{\R^2}   (|x-y_1|^{0-}+\la x-y_1\ra^\gamma) \la y_1\ra^{-2-2\gamma-} (|y_1-y|^\gamma +|y_1-y|^{-1}) dy_1 \les \la x\ra^\gamma\la y\ra^\gamma.
	$$ 
	
\end{lemma}

\begin{lemma}\label{lem:spa_int1}
	The following integral bound holds
	$$
	\int_{\R^2} \big(1+|y_1-y_2|^{-1}\big) \la y_2\ra^{-2-}  \big(1+|y_2-y |^{-1}\big) dy_2  \les \big(1+|y_1-y |^{0-}\big).
	$$

\end{lemma}

\begin{proof}
	
	The proof follows using Lemma~\ref{lem:EGcor} provided we show
	$$
		\int_{\R^2} \frac{\la y_2\ra^{-2-}}{|y_1-y_2||y_2-y|}  dy_2 
		\les 1+|y_1-y |^{0-}.
	$$
	We note that
	$$
		\frac{1}{|y_1-y_2||y_2-y|} \les \frac{1}{|y_1-y_2||y_2-y|^{1+}}+\frac{1}{|y_1-y_2| |y_2-y|^{1-}},
	$$
	and Lemma~\ref{lem:EGcor} finishes the proof.
	
\end{proof}

\begin{large}
	\noindent
	{\bf Acknowledgment. \\}
\end{large}
The authors would like to thank Fritz Gesztesy for a careful reading and providing helpful comments on an preliminary version of this paper.

\end{document}